\documentclass[letterpaper, 11pt,  reqno]{amsart}
\usepackage{comment}
\usepackage{bbm}
\usepackage{amsmath,amssymb,amscd,amsthm,amsxtra, esint}
\usepackage{amsfonts,latexsym}
\usepackage{eucal}
\usepackage{mathabx}
\usepackage{mathtools}
\usepackage{bbm}
\usepackage{cite}
\usepackage[margin=1.2in,marginparwidth=1.5cm, marginparsep=0.5cm]{geometry}

\usepackage[implicit=true]{hyperref}
\allowdisplaybreaks[2]

\newtheorem{theorem}{Theorem} [section]

\newtheorem{remark}[theorem]{Remark}

\newtheorem{definition}[theorem]{Definition}

\DeclareMathOperator*{\supp}{supp}

\newcommand{\Z}{\mathbb{Z}}
\newcommand{\R}{\mathbb{R}}
\newcommand{\N}{\mathbb{N}}
\newcommand{\C}{\mathcal{C}}
\newcommand{\T}{\mathbb{T}}

\let\Re=\undefined\DeclareMathOperator*{\Re}{Re}
\let\Im=\undefined\DeclareMathOperator*{\Im}{Im}

\newcommand{\les}{\lesssim}
\newcommand{\ges}{\gtrsim}

\newcommand{\cu}{\widecheck{u}}

\numberwithin{equation}{section}
\numberwithin{theorem}{section}

\usepackage{color}

\newtheorem{thm}{Theorem}[section]
\newtheorem{lem}[thm]{Lemma}
\newtheorem{prop}[thm]{Proposition}
\newtheorem{cor}[thm]{Corollary}
\newcommand{\abrac}[1]{\left\langle #1 \right\rangle}
\theoremstyle{definition}

\theoremstyle{remark}

\def\supp{\mathop{\rm supp}\nolimits}

\def\ind{\mathbbm{1}}
\def\indh{\mathbbm{1}_{T,R}^{\text{high}}}
\def\indl{\mathbbm{1}_{T,R}^{\text{low}}}

\renewcommand{\bar}[1]{\overline{#1}}

\newcommand{\del}{\partial}

\newcommand{\sgn}{{\mbox{sgn}}}

\newcommand{\Nmax}{N_{\text{max}}}

\renewcommand{\hat}{\widehat}
\renewcommand{\Tilde}{\widetilde}

\mathtoolsset{showonlyrefs}
\numberwithin{equation}{section}

\title{Unconditional well-posedness for the MMT equation  on the torus}

\author{Mahendra Panthee}
\address{Department of Mathematics, University of Campinas, Brazil}
\email{mpanthee@unicamp.br}
\author{James Patterson}
\address{School of Mathematics, University of Birmingham, UK}
\email{jxp277@student.bham.ac.uk}
\author{Yuzhao Wang}
\address{School of Mathematics, University of Birmingham, UK}
\email{y.wang.14@bham.ac.uk}

\begin{document}

\begin{abstract} 
We consider the initial value problem (IVP) for a two-parameter family of
derivative nonlinear Schr\"odinger equations on the torus, known as the Majda--McLaughlin--Tabak (MMT) model arising
in weak wave turbulence theory.  For positive derivative order, we show that
the flow map is not $C^3$ at the origin.  Using an enhanced energy method, we
prove unconditional local well-posedness in Sobolev spaces.  
At the energy regularity, conservation of the Hamiltonian and a
mass-type quantity yields unconditional global well-posedness.
\end{abstract}

\maketitle

\section{Introduction}
In this article, we consider the initial value problem (IVP) for the following two-parameter family of derivative nonlinear Schr\"odinger (NLS) equations
\begin{equation} \label{MMTEq}
\begin{cases}
i\del_t u + (-\partial_x^2)^{\frac{\alpha}{2}}u = \nu D_x^{2\beta} (|u|^2 u),\\
u(x, 0) = u_0(x),
\end{cases}\qquad x \in \T, t\in\R,
\end{equation}
where $u = u(x,t)$ is a complex function, $\alpha > 0$,  $\nu \in \{-1,1\}$, $\beta \geq 0$ and the operator $D_x^{2\beta}$ is defined via Fourier transform $\hat{D_x^{2\beta}u}(k) = |k|^{2\beta}\hat{u}(k)$. This equation is equivalent to the Majda--McLaughlin--Tabak (MMT) model
\begin{equation} \label{MMT}
\begin{cases}
i\del_t v + (-\partial_x^2)^{\frac{\alpha}{2}}v =\nu D_x^{\beta} (|D_x^\beta v|^2 D_x^\beta v),\\
v(x, 0) = v_0(x),
\end{cases}\qquad x \in \T, t\in\R,
\end{equation}
introduced in \cite{MMT-97} to test the predictions for weak wave turbulence theory on a computationally tractable system with a tunable dispersion and nonlinearity, see \cite{ZDP-04, ZGPD-01}. We can observe that the transformation 
\begin{equation} \label{MMTOldToNew}
    u(x,t) = D_x^\beta v(x,t),
\end{equation}
converts \eqref{MMT} into \eqref{MMTEq} with
$u_0=D_x^\beta v_0$.  When $\beta>0$, this transformation produces a
zero-mean function $u$, and the zero-mean subspace is invariant under
\eqref{MMTEq}.  Conversely, every zero-mean solution of \eqref{MMTEq}
determines a solution of \eqref{MMT} after fixing the constant zero mode of
$v$\footnote{In Fourier variables, the inverse is
$\widehat v(k)=|k|^{-\beta}\widehat u(k)$ for $k\ne0$, while
$\widehat v(0)=\widehat v_0(0)$ is constant in time.}.
Thus, for $\beta>0$, the MMT equation is equivalent to the zero-mean sector
of \eqref{MMTEq}, with Sobolev indices shifted by $\beta$ and with an
independent constant mode for $v$; when $\beta=0$, the transformation is the
identity.  We refer to \eqref{MMTEq} as the transformed MMT equation and use
it throughout the analysis.

In our previous work \cite{PPW-26} we considered the IVP \eqref{MMT} posed on $\R$ and obtained sharp well-posedness results for data in classical Sobolev spaces. However, in the periodic case, the method employed in \cite{PPW-26} is not suitable for $\beta > 0$. It is well known that the contraction mapping principle fails for a large class of nonlinear dispersive equations \cite{MST-01}; this includes the derivative nonlinear Schr\"odinger equation on $\T$ in its original form \cite{HT-99}. The same obstruction occurs for \eqref{MMTEq}, as shown by the counterexample in Section \ref{ill-p}. We instead use an enhanced energy method that combines frequency-dependent time localization, Strichartz estimates, and energy estimates; see \cite{IKT-08, KK-03, KT-07, KT-03, MolinetUncondKdV, MPV-18, MV-15, MolinetTanakaUncond} and the references therein.

\begin{definition} \label{DistributionalSolution}
Let $T>0$ and $s>\frac12$. We say that $u\in L^{\infty}([0, T]; H^s(\T))$ is a solution to the IVP \eqref{MMTEq} with initial data $u_0\in H^s(\T)$ if $u$ satisfies \eqref{MMTEq} in the distributional sense, i.e., 
$$\int_0^{T}\int_{\T} (\overline{i\partial_t\varphi+(-\Delta)^{\frac{\alpha}2}\varphi})u -  \nu (\overline{D_x^{2\beta} \varphi})|u|^2u \,dxdt +\int_{\T}\overline\varphi(\cdot, 0) u_0 \,dx = 0,$$
holds for any test function $\varphi\in C_0^{\infty}([0, T] \times \T)$.

\end{definition}

\begin{definition} \label{UUDef}
We  say that the IVP  \eqref{MMTEq} is unconditionally locally well-posed in $H^s(\T)$ if the following hold:
\begin{itemize}
\item For any initial data $u_0 \in H^s(\T)$ there exist $ T=T(\|u_0\|_{H^s})>0$ and a solution $u \in C([0, T]; H^s(\T))$ to \eqref{MMTEq} emanating from $u_0$.
\item The solution $u$ is unique in the class $L^{\infty}([0, T]; H^s(\T))$.
\item  For any $R>0$, the solution-map $u_0\mapsto u$ is continuous from a ball of $H^s(\T)$ with radius $R$ centered at the origin into $C([0, T(R)]; H^s(\T))$.

\end{itemize}

\end{definition}

Now, we are in position to state the first main result of this work.

\begin{thm} \label{MainTheorem}
Let $1 < \alpha \leq 2$, $0 \leq \beta  \leq \alpha/4$, and
\[
s \geq \max\left\{\frac32\beta + \frac{2-\alpha}{4},
\beta + 1 -\frac\alpha2, \frac12+\right\}.
\]
Then the IVP \eqref{MMTEq} is unconditionally locally well-posed for initial
data $u_0\in H^s(\T)$.
\end{thm}

\begin{remark}
The result of Theorem \ref{MainTheorem} concerns unconditionally unique
solutions according to Definition \ref{UUDef}; there are already both conditional and unconditional
results for fNLS ($\beta = 0$). 
For both the real line $\R$ and torus $\T$
cases, Cho--Hwang--Kwon--Lee \cite{YGSS-WellandIll} (see also
\cite{DET-ExistTheoryfNLS} for the torus) proved conditional well-posedness for
$s \geq (2-\alpha)/4$, with sharpness on $\R$. Their argument uses the Fourier
restriction norm method and a contraction in an $X^{s,b}$ space; see
\eqref{BourgainDef}.
Kishimoto \cite{Kishimoto} proved unconditional well-posedness for $s\ge1/6$ and $s>(2-\alpha)/4$ by an infinite normal-form reduction.

We remark that the situation considered in Theorem \ref{MainTheorem} is much more involved than the real-line case obtained in \cite{PPW-26} as well as the $\beta = 0$ case in \cite{Kishimoto}. 
The argument in \cite{PPW-26} is
based on a contraction and therefore produces a $C^3$ flow map, whereas
Section \ref{ill-p} rules out that mechanism on the torus when $\beta>0$.
For the unconditional result, Kishimoto \cite{Kishimoto} considered the $\beta= 0$ case by an infinite normal-form reduction. 
Extending that argument to $\beta>0$ is substantially more difficult because each normal-form step must account for the multiplier $D_x^{2\beta}$. 
\end{remark}

For sufficiently smooth solutions, \eqref{MMTEq} has the Hamiltonian
\begin{equation}\label{Energy}
    H(u) := \int_\T \Big(|D_x^{\frac\alpha2 - \beta} u |^2 - \frac{\nu}{2}|u|^4 \Big) \,dx.
\end{equation}
It also has the conserved mass-type quantity
\begin{equation}\label{MassType}
\mathcal M_\beta(u):=
|\widehat u(0)|^2+\|D_x^{-\beta}P_{\ne0}u\|_{L^2}^2,
\end{equation}
which is equivalent to $\|u\|_{H^{-\beta}}^2$.  Under the transformation
\eqref{MMTOldToNew}, when $\beta>0$ one has
$\mathcal M_\beta(u)=\|P_{\ne0}v\|_{L^2}^2$; adding the independently
conserved quantity $|\widehat v(0)|^2$ gives the full $L^2$ mass of $v$.
When $\beta=0$, the transformation is the identity and
$\mathcal M_0(u)=\|v\|_{L^2}^2$.
The energy regularity is
\[
s_E:=\frac{\alpha}{2}-\beta.
\]
For this conservation law to imply global well-posedness directly from
Theorem \ref{MainTheorem}, the index $s_E$ must belong to the local
well-posedness range.  This gives
\[
\beta<\frac{\alpha-1}{2},
\qquad
\beta\le\frac{3\alpha-2}{10}.
\]
The argument below follows the standard energy-space continuation framework;
see, in particular, \cite[Section 6]{MolinetTanakaUncond}, where conservation
laws are first extended to energy-class solutions and are then combined with
a local theory whose lifespan is controlled by the energy norm.

\begin{thm}[Global well-posedness at the energy regularity]
\label{GlobalTheorem}
Let $1<\alpha\le2$.  Assume either
\begin{enumerate}
\item $\nu=-1$ and
\[
0\le\beta<\frac{\alpha-1}{2},
\qquad
\beta\le\frac{3\alpha-2}{10};
\]
\item $\nu=1$ and
\[
0\le\beta<\frac{\alpha-1}{4}.
\]
\end{enumerate}
Then the IVP \eqref{MMTEq} is unconditionally globally well-posed in
$H^{s_E}(\T)$, where $s_E=\frac{\alpha}{2}-\beta$.  More precisely, for every
$u_0\in H^{s_E}(\T)$ there is a unique solution
\[
u\in C(\R;H^{s_E}(\T)),
\]
and the solution map is continuous from bounded subsets of
$H^{s_E}(\T)$ into $C([-T,T];H^{s_E}(\T))$ for every $T>0$.
\end{thm}

\begin{proof}
We first verify that the energy index is covered by Theorem
\ref{MainTheorem}.  The condition $s_E>1/2$ is equivalent to
$\beta<(\alpha-1)/2$, while
\[
s_E\ge \beta+1-\frac{\alpha}{2}
\]
follows from the same restriction.  Moreover,
\[
s_E\ge \frac32\beta+\frac{2-\alpha}{4}
\quad\Longleftrightarrow\quad
\beta\le\frac{3\alpha-2}{10}.
\]
Moreover, since $1<\alpha\le2$,
\[
\frac{\alpha-1}{2}\le\frac{\alpha}{4}.
\]
Thus the strict condition $\beta<(\alpha-1)/2$ also implies the structural
restriction $\beta\le\alpha/4$ in Theorem \ref{MainTheorem}.  In the focusing
case, the stronger condition $\beta<(\alpha-1)/4$ implies all three
restrictions.
Thus Theorem \ref{MainTheorem} gives unconditional local well-posedness in
$H^{s_E}$.

We next prove conservation of \eqref{Energy} and \eqref{MassType}.  Let $u$
be a smooth solution.
When $\beta>0$, the zero Fourier mode of \eqref{MMTEq} is constant.  On the
nonzero modes, applying $D_x^{-2\beta}$ to the equation gives
\[
iD_x^{-2\beta}\partial_tu+D_x^{\alpha-2\beta}u
=\nu P_{\ne0}(|u|^2u).
\]
Since $\partial_t\widehat u(0)=0$, the projection $P_{\ne0}$ may be omitted
when paired with $\partial_tu$.  Consequently,
\begin{align*}
\frac{d}{dt}H(u(t))
&=2\Re\left\langle
D_x^{\alpha-2\beta}u-\nu|u|^2u,\partial_tu
\right\rangle_{L^2}\\
&=2\Re\left\langle
-iD_x^{-2\beta}\partial_tu,\partial_tu
\right\rangle_{L^2}=0.
\end{align*}
Taking instead the imaginary part of the $L^2$ pairing of this equation with
$u$ gives
\[
\begin{split}
0
&=\Im\left\langle iD_x^{-2\beta}\partial_tu,u\right\rangle_{L^2}\\
&=\Re\left\langle D_x^{-2\beta}\partial_tu,u\right\rangle_{L^2}
=\frac12\frac{d}{dt}
\|D_x^{-\beta}P_{\ne0}u(t)\|_{L^2}^2,
\end{split}
\]
because both
$\langle D_x^{\alpha-2\beta}u,u\rangle_{L^2}$ and
$\langle |u|^2u,u\rangle_{L^2}$ are real.  Together with the constancy of
$\widehat u(0)$, this proves conservation of \eqref{MassType}.
For $\beta=0$, the same computation applies with $D_x^{-2\beta}$ equal to
the identity.  To pass from smooth solutions to energy-class solutions, we
argue as in \cite[Section 6]{MolinetTanakaUncond}.  When $\beta>0$, the
Galerkin approximations \eqref{Galerkin} obey the same identities: since
$u_K$ and $\partial_tu_K$ lie in the range of the self-adjoint projection
$\Pi_K$, the projection may be removed in each of the pairings above.
Moreover, $u_K\to u$ in $C([0,T];H^{s_E})$.  When $\beta=0$, one instead
approximates the initial datum by smooth functions and uses the continuous
dependence supplied by the contraction argument in Section \ref{sec-7}.
Since $s_E>1/2$, the maps
\[
u\longmapsto \|D_x^{s_E}u\|_{L^2}^2
\quad\text{and}\quad
u\longmapsto\|u\|_{L^4}^4,
\qquad
u\longmapsto\mathcal M_\beta(u)
\]
are continuous on $H^{s_E}(\T)$.  Passing to the limit therefore shows that
the local solution constructed in Section \ref{sec-7} satisfies
\[
H(u(t))=H(u_0),
\qquad
\mathcal M_\beta(u(t))=\mathcal M_\beta(u_0)
\]
throughout its interval of existence.

For the defocusing sign $\nu=-1$, the Hamiltonian is directly coercive:
\[
H(u)=\|D_x^{s_E}u\|_{L^2}^2+\frac12\|u\|_{L^4}^4.
\]
By H\"older's inequality on $\T$,
\[
\|u\|_{L^2}^2
\le (2\pi)^{1/2}\|u\|_{L^4}^2
\le 2\sqrt{\pi H(u_0)}.
\]
It follows that, for every time in the local existence interval,
\begin{equation}\label{EnergyCoercivity}
\|u(t)\|_{H^{s_E}}^2
\lesssim H(u_0)+2\sqrt{\pi H(u_0)}.
\end{equation}
For the focusing sign $\nu=1$, assume $\beta<(\alpha-1)/4$.  Sobolev
embedding and interpolation give
\begin{equation}\label{L4bound}
\|u\|_{L^4}
\lesssim \|u\|_{H^{1/4}}
\lesssim
\|u\|_{H^{-\beta}}^{1-\vartheta}
\|u\|_{H^{s_E}}^\vartheta,
\qquad
\vartheta=\frac{4\beta+1}{2\alpha}<\frac12.
\end{equation}
By conservation of \eqref{MassType} and Young's inequality, for every
$\varepsilon>0$,
\[
\begin{split}
\|u(t)\|_{L^4}^4
&\lesssim
\mathcal M_\beta(u_0)^{2(1-\vartheta)}
\|u(t)\|_{H^{s_E}}^{4\vartheta}\\
&\le \varepsilon\|u(t)\|_{H^{s_E}}^2
+C_\varepsilon
\mathcal M_\beta(u_0)^{\frac{2(1-\vartheta)}{1-2\vartheta}}.
\end{split}
\]
Since $\mathcal M_\beta(u_0)$ controls the zero mode and
$\|D_x^{s_E}u\|_{L^2}$ controls every nonzero mode,
\[
\begin{split}
\|u(t)\|_{H^{s_E}}^2
&\lesssim
\mathcal M_\beta(u_0)+\|D_x^{s_E}u(t)\|_{L^2}^2\\
&=\mathcal M_\beta(u_0)+H(u_0)
+\frac12\|u(t)\|_{L^4}^4\\
&\le C\big(\mathcal M_\beta(u_0)+|H(u_0)|\big)
+C\varepsilon\|u(t)\|_{H^{s_E}}^2
+C_\varepsilon
\mathcal M_\beta(u_0)^{\frac{2(1-\vartheta)}{1-2\vartheta}}.
\end{split}
\]
Choosing $\varepsilon>0$ so that the penultimate term can be absorbed into
the left-hand side gives
\begin{equation}\label{FocusingEnergyCoercivity}
\sup_t\|u(t)\|_{H^{s_E}}^2
\le C\big(|H(u_0)|,\mathcal M_\beta(u_0)\big)
\end{equation}
on every local existence interval.  Notice that
$\beta<(\alpha-1)/4$ implies all the local-theory restrictions verified
above.

The local existence time in Theorem \ref{MainTheorem} depends only on the
$H^{s_E}$ norm of the data.  The uniform bounds
\eqref{EnergyCoercivity} and \eqref{FocusingEnergyCoercivity}, in their
respective cases,
therefore allow us to restart the local theory after a time interval whose
length is bounded below solely in terms of the conserved quantities.
Iteration gives a
solution for all positive and negative times.  Unconditional uniqueness and
continuous dependence on every compact time interval follow by iterating the
corresponding local statements finitely many times.  This is the same
continuation principle used in \cite[Section 6]{MolinetTanakaUncond}.
\end{proof}

\begin{remark}
The Hamiltonian controls only $H^{s_E}$.  It therefore gives the global
result above at the energy regularity, but by itself it does not preclude
growth of an $H^s$ norm when $s>s_E$. 
\end{remark}

Finally, reversing \eqref{MMTOldToNew} gives the following consequence for
the MMT model \eqref{MMT}.
\begin{thm}
Let $1<\alpha\le2$ and $0\le\beta\le\alpha/4$.  If
\[
s\ge
\max\left\{
\frac52\beta+\frac{2-\alpha}{4},
2\beta+1-\frac{\alpha}{2},
\beta+\frac12+
\right\},
\]
then the IVP \eqref{MMT} is unconditionally locally well-posed for initial
data $v_0\in H^s(\T)$.
\end{thm}
We also have the following global well-posedness.
\begin{cor}
\label{GlobalMMT}
Let $1<\alpha\le2$, and assume either
\begin{enumerate}
\item $\nu=-1$ and
\[
0\le\beta<\frac{\alpha-1}{2},
\qquad
\beta\le\frac{3\alpha-2}{10};
\]
\item $\nu=1$ and
\[
0\le\beta<\frac{\alpha-1}{4}.
\]
\end{enumerate}
Then the MMT equation \eqref{MMT} is unconditionally globally well-posed in
$H^{\alpha/2}(\T)$.
\end{cor}

\begin{proof}
If $\beta=0$, then $u=v$, and the conclusion is exactly Theorem
\ref{GlobalTheorem}.  If $\beta>0$, set $u=D_x^\beta v$.  Theorem
\ref{GlobalTheorem} gives global control of $u$ in
$H^{\alpha/2-\beta}$, which is equivalent to control of the nonzero Fourier
modes of $v$ in $H^{\alpha/2}$.  In this case the zero mode of $v$ is constant
in time, so the conclusion follows from the inverse transformation described
after \eqref{MMTOldToNew}.
\end{proof}

\begin{remark}
It is worth mentioning the work in \cite{KKO-25}, where the authors established a necessary and sufficient condition on
the nonlinearity for the well-posedness of derivative fractional nonlinear
Schr\"odinger equations, extending earlier results in \cite{KO-25} for
semilinear Schr\"odinger equations on $\T$. The modified energy method in
\cite{KKO-25} treats polynomial nonlinearities containing derivatives of at
most first order. It is not immediate that this method extends to the
fractional derivative in \eqref{MMTEq}.
\end{remark}

\begin{remark}
It is useful intuition when working with equations of this form to make a comparison to the work of \cite{TKMO-26}. Through methods of norm inflation it is demonstrated that on $\T$ the equation
\begin{equation}
i\del_t u -D_x^\alpha u = u^k D_x^\beta u
\end{equation}
is ill-posed for all $\beta > 0$, even in the sense of definition \ref{UUDef}. We expect that a modification of their proof could also demonstrate a similar ill-posedness for $\beta > 0$ for
\begin{equation}
i\del_tu + (-\del_x^2)^{\alpha/2}u = D_x^{2\beta}(u^3).
\end{equation}
The Fourier half-space plays a crucial role in the argument of
\cite{TKMO-26}.  Such a restriction is unavailable for
$D_x^{2\beta}(|u|^2u)$, because both $u$ and $\overline u$ occur.
Consequently, that argument does not apply directly to \eqref{MMTEq}.
\end{remark}

\medskip
  
\noindent
{\bf Organisation of the paper:} 
The remainder of this work is organised as follows. 
Section~\ref{prelim} introduces the function spaces and preliminary estimates needed throughout the paper. 
In Section \ref{ill-p} we provide a counterexample to $C^3$ regularity of
the flow map. In Section~\ref{sec-4}, we recall properties of the resonance
relation and prove Strichartz estimates.
Section~\ref{sec-5} is devoted to obtaining an {\em a priori} estimate for the solution while in Section \ref{sec-6} we find an estimate for the difference of the solutions. 
Proof of the main result is provided in Section~\ref{sec-7}.  \\

\section{Preliminaries}\label{prelim}

In this section, we introduce definitions, notations and preliminary results. For $x,y \in \R$, we write $x \les y$ to denote there exists a constant $C > 0$ such that $x \leq Cy$; when $0 < C < 1$ we further write $x \ll y$. When $x \les y$ and $y \les x$ we write $x \sim y$.
It will become necessary to sum over dyadic numbers, when doing so, we will capitalize the relevant index, for example $N, L \in 2^{\Z}$. Lastly, for $\varepsilon > 0$ taken arbitrarily small, we will often write $x +$ as shorthand for $x+ \varepsilon$. $x-$ is analogously defined.

For the Fourier transform of $f \in \mathcal{S}'$ we follow the usual convention writing $\hat{f}$ (or $\mathcal{F}[f]$ when appropriate) to denote the Fourier transform
\begin{equation}
\hat{f}(k,\tau) := \int_{\R} \int_\T e^{- ixk} e^{-it\tau} f(x,t) \, dxdt.
\end{equation}
$\mathcal{F}_x, \mathcal{F}_t$ represent the Fourier transform only in their respective variables.

Let $\varphi \in C^\infty_c(\R)$ be such that $\varphi : \R \rightarrow [0,1]$ is a smooth radial cut-off function for which $\varphi(k) \equiv 1$ for $|k| \leq 1$ and $\supp \varphi \subseteq [-2,2]$. We then define the smooth annulus cut-off function $\phi:\R \rightarrow[0,1]$ by $\phi(k) := \varphi(|k|/2) - \varphi(|k|) $, which we extend to higher scales with $\phi_N(k):= \phi(k/N)$; in this way $\phi_N$ is supported around $|k| \sim N$. For $N \geq 1$ we then define $P_N$ as $\hat{P_Nf}(k) = \phi_N(k)\hat{f}(k)$, the frequency projection down to frequencies $|k| \sim N$; furthermore we define $P_0f = \hat{f}(0)$.
Sobolev space $H^s_x$ on the torus $\T$ is then defined as
\begin{equation}
\|u\|_{H^s_x} := \bigg( \sum_{\xi \in \Z} \abrac{k}^{2s}|\hat{u}(k)|^2 \bigg)^{1/2} \sim \bigg( \sum_{\substack{N\in 2^{\N_0}}} N^{2s}\|P_{N}u\|_{L^2_x}^2\bigg)^{1/2} + \|P_0 u\|_{L^2_x}
\end{equation}
The Sobolev space on $\R$ is defined analogously, where $H^s_t$ denotes the Sobolev space in the $t$ variable. Note when $H^s$ appears with no subscript, this will always refer to $H_x^s$. 
We let $S(t)$ denote the linear propagator of \eqref{MMTEq}. That is
\begin{equation} \label{LinProp} 
\mathcal{F}_x[S(t)u_0](k) := e^{it|k|^\alpha}\hat{u}_0(k).
\end{equation}
To make use of the modulation, for $s, b\in \R$, we introduce the well-known Fourier restriction norm method $X^{s,b}$ with norm defined by
\begin{equation} \label{BourgainDef}
\|u\|_{X^{s,b}} := \|S(-t)u\|_{H^b_t H^s_x} \sim \bigg(\sum_{L \geq 1} L^{2b} \|Q_L u \|_{L^2_t H^s_x}^2 \bigg)^{1/2}
\end{equation}
where $Q_L$ is the modulation projection. For $L > 1$ this is defined as $\hat{Q_L u}(k,\tau) = \phi_L(\tau-|k|^\alpha) \hat{u}(k, \tau)$, while
\begin{equation}
    \hat{Q_1u}(k,\tau) := \sum_{0 < L \leq 1} \phi_L(\tau - |k|^\alpha)\hat{u}(k,\tau).
\end{equation}
We also define
$Q_{ * K} := \sum_{L * K}Q_L$ where $*$ is any relational operator.

A valuable tool in demonstrating the continuity of the solution map with respect to initial data is the following weighted Sobolev space. For each $N \in 2^\N$, we define the frequency envelope $\omega_N$ as a dyadic sequence satisfying $1 \leq \omega_N \leq \omega_{2N} \leq \delta \omega_N$ for some $1 < \delta \leq 2$.

We then define weighted Sobolev space $H^s_\omega$ through the norm
\begin{equation}\label{freqEnvelope}
\|u\|_{H^s_\omega} := \left(\sum_{N \geq 1} \omega_N^2 N^{2s}\|P_N u\|_{L^2_x}^2 \right)^{\frac{1}{2}}+ \|P_0u\|_{L^2_x}.
\end{equation}
This space will be included in many estimates, though in most proofs we assume $\omega_N \equiv 1$ as the proof does not change. We similarly define $X^{s,b}_\omega$ through the norm
\begin{equation}
\| u\|_{X^{s,b}_\omega} := \|S(-t)u\|_{H^b_t H^s_\omega}.
\end{equation}
The weighted space $H^s_\omega$ is used to prevent concentration of the
Sobolev norm at arbitrarily high frequencies. Such spaces were used in
\cite{KT-03} as an alternative to a Bona--Smith argument.

We will need to consider Fourier restriction norm method localised to the time interval $[0,T]$. In general, for $B$ a normed space of space-time functions, we define $B_T$ as the restriction space
\begin{equation}
\|u\|_{B_T} := \inf\{ \|v\|_{B} : u(t) = v(t) \text{ for } t\in[0,T] \}.
\end{equation}
For the restriction of a Fourier restriction norm method in the presence of a frequency envelope as in \eqref{freqEnvelope}, we further write $X^{s,b}_{\omega,T}$.
The value of these Fourier restriction norm methods, is that they introduce a simple way to exploit modulation for derivative gain. Furthermore, they behave well when applied to the Duhamel formulation, which for \eqref{MMTEq} is
\begin{equation} \label{Duhamel}
u(t) = S(t)u_0 -i\nu\int_0^t S(t-t')D_x^{2\beta}(|u|^2u)(t') \,dt'.    
\end{equation}
These Fourier restriction norm methods satisfy standard but very useful linear estimates on both linear and Duhamel terms.
\begin{lem} \label{StandardLinear}
Let $0 < T \leq 1$, $u_0 \in H^s$, $F \in L^2_TH^s$, the following linear estimates hold
\begin{equation} \label{LinearS}
\|S(t)u_0\|_{X^{s,1}_{\omega,T}} \les \|u_0\|_{H^{s}_\omega}
\end{equation}
and
\begin{equation} \label{LinearD}
\bigg\|\int_0^t S(t-t')F(t') \,dt' \bigg\|_{X^{s,1}_{\omega,T}} \les  \|F\|_{L^2_T H^s_{\omega}}.
\end{equation}
\end{lem}
\begin{proof}
Proof is standard, see for instance \cite{GTV}.
\end{proof}
An issue that arises from this use of Fourier restriction norm method is that proving any kind of estimate requires functions defined for all $t \in \R$. Luckily, this is just something that requires technical finesse rather than being a fundamental roadblock. We circumvent these problems by defining extensions beyond $[0,T]$ to $\R$ in the same way as in \cite{MolinetTanakaUncond}. Viewing $\varphi \in C^\infty_c(\R)$ as defined above now as a time cut-off. We define an extension of $u$ with the operator $\rho_T$ as 
\begin{equation} \label{extension}
\rho_T(u) := S(t)\varphi(t) S(-\mu_T(t)) u(\mu_T(t)),
\end{equation}
where $\mu_T$ is the continuous piecewise function
\[
\mu_T(t) := 
\begin{cases}
0 \, &\text{ if } \hspace{.5cm} t \not\in [0,2T] \\
t \, &\text{ if } \hspace{.5cm}t \in [0,T] \\
2T - t &\text{ if } \hspace{.5cm}t\in [T,2T].
\end{cases}
\]
Observe that for $0 < T \leq 1$ and $t \in [0,T]$
\[ u(t) = \rho_T(u)(t),\]
hence, this is a true extension of $u$ beyond the original time interval. Defining $Z^s_\omega := L^\infty_t H^s_\omega \cap X^{s-2\beta,1}_\omega$, the benefit of this extension is revealed in the following lemma.
\begin{lem} \label{rhoBounded}
Let $0 < T \leq 1$ and $s \in \R$. Then
\[ \rho_T:X_{\omega,T}^{s-2\beta,1}\cap L^\infty_T H^s_\omega \rightarrow Z^s_\omega \]
\ is a bounded linear operator, with
\[ \|\rho_T(u)\|_{L^\infty_t H^s_\omega} + \|\rho_T(u)\|_{X^{s-2\beta,1}_\omega} \lesssim \|u\|_{L^\infty_T H^s_\omega} + \|u\|_{X^{s-2\beta,1}_{\omega,T}}\]
for all $u \in X^{s-2\beta,1}_{\omega, T}\cap L^\infty_TH^s_\omega$.
Moreover
\begin{equation} 
\|\rho_T(u)\|_{L^\infty_t H^s_\omega} \lesssim \|u\|_{L^\infty_T H^s_\omega}
\end{equation}
for $u \in L^\infty_T H^s_\omega$.
\end{lem}
\begin{proof}
We assume $\omega_N \equiv 1$, as the proof is independent of the frequency envelope. The unitary nature of $S(t)$ in $H^s$ immediately gives
\begin{equation}
\|\rho_T(u)\|_{L^\infty_t H^s} \les \|u(0)\|_{H^s} + \|u(\cdot)\|_{L^\infty_T H^s} + \|u(T-\cdot)\|_{L^{\infty}_T H^s} \les \|u\|_{L^\infty_T H^s}.
\end{equation}
It is therefore sufficient to show estimates in $X^{s,b}$. By definition \eqref{BourgainDef}
\begin{equation}
\begin{split}
\|\rho_T(u)\|_{X^{s-2\beta,1}} & \les \|  \varphi(\cdot) S( - \mu_T(\cdot))u(\mu_T(\cdot))\|_{H^1_t H^{s-2\beta}_x} \\
& \les \| \varphi(\cdot) S( - \mu_T(\cdot))u(\mu_T(\cdot))\|_{L^2_t H^{s-2\beta}_x} + \| \del_t[\varphi(\cdot) S( - \mu_T(\cdot))u(\mu_T(\cdot))]\|_{L^2_t H^{s-2\beta}_x}.
\end{split}
\end{equation}
By $X_T^{s-2\beta,1} \hookrightarrow C([0,T]; H^{s-2\beta})$ and definition of $\mu_T$
\begin{equation}
\begin{split} \label{rhoBounded1}
\| \varphi(\cdot) S( - \mu_T(\cdot))u(\mu_T(\cdot))\|_{L^2_t H^{s-2\beta}_x} &\les \|u(0)\|_{H^{s-2\beta}_x} +\|S(-\,\cdot)\varphi(\cdot) u(\cdot)\|_{L^2_T H^{s-2\beta}_x} \\
&\leq \|u\|_{L^\infty_T H^{s-2\beta}_x} + \|u\|_{X_T^{s-2\beta,0}}
\end{split}
\end{equation}
and 
\begin{equation}
\begin{split}
\| \del_t[\varphi(\cdot) S( - \mu_T(\cdot))& u(\mu_T(\cdot))]\|_{L^2_t H^{s-2\beta}_x} \\ &= \|\varphi_t S( - \mu_T(\cdot))u(\mu_T(\cdot))\|_{L^2_t H^{s-2\beta}_x}  + \|S(-\, \cdot)(\del_t - i(-\Delta)^{\frac{\alpha}{2}})u(\cdot)\|_{L^2_T H^{s-2\beta}_x} \\
& \les \|u(0)\|_{H^{s-2\beta}_x} + \|u\|_{X^{s-2\beta,0}_T} + \inf_{v}\| |\tau - |k|^\alpha| \abrac{k}^{s-2\beta} \hat{v} (k,\tau)\|_{L^2_{\tau\xi}} \\
& \les \|u\|_{L^\infty_T H^{s-2\beta}} + \|u\|_{X^{s-2\beta,1}_T}.
\end{split}
\end{equation}
The bound on the term including $\varphi_t$ follows from the same argument for \eqref{rhoBounded1}. This completes the proof.
\end{proof}
In this way we can extend a function $u$ defined on $[0,T]$ to $\R$, while recovering bounds in terms of the original function and norm.

We denote by $\ind_T$ the characteristic function of the interval $[0,T]$. This will naturally arise in the computation of time integrals; when replacing functions with their extensions $\cu$, we require an indicator function to ensure value of the time integral remains unchanged. A new technical hindrance then arises when we plan to make use of modulation, the operator $Q_{\gtrsim L}$ does not commute with $\ind_T$. To bypass this, we further decompose
\begin{equation} \label{indDecomp}
\ind_T = \ind_{T,R}^{\text{high}} + \ind_{T,R}^{\text{low}} \hspace{.5cm} \text{with}\hspace{.5cm} \mathcal{F}_t(\ind_{T,R}^{\text{low}})(\tau) = \varphi(\tau/R)\mathcal{F}_t(\ind_{T})(\tau)
\end{equation}
for some choice of $R$.
The value of this decomposition is that, for $R$ not too large, $Q_L$ does essentially commute with $\ind_{T,R}^{\text{low}}$
\begin{lem} \label{QLlow}
For $T > 0, R> 0$ and $L \gg R$, it holds
\begin{equation}
\|Q_L(\ind_{T,R}^{\text{low}}u)\|_{L^2_{tx}} \lesssim \|Q_{\sim L}u\|_{L^2_{tx}}
\end{equation}
for all $u \in L^2(\R \times \T)$.
\end{lem}
\begin{proof}
    See Lemma 2.5 in \cite{MV-15} or Lemma 3.7 in \cite{MolinetUncondKdV}.
\end{proof}
Assuming that exploiting modulation can sufficiently estimate terms in which $\indl$ arises, it then remains to find a way to bound terms containing $\ind_{T,R}^{\text{high}}$. Fortunately, the following means we can bound these terms independently of modulation.

\begin{lem}\label{time-multi}
For any $R > 0$ and $T > 0$, it holds 
\begin{equation} \label{indL1Bound}
\|\ind_{T,R}^\text{high}\|_{L^1_t} \lesssim T \wedge R^{-1}
\end{equation}
and
\begin{equation} \label{indLinftyBound}
\|\ind_{T,R}^\text{high}\|_{L^\infty_t} +\|\ind_{T,R}^\text{low}\|_{L^\infty_t} \lesssim 1.
\end{equation}
Furthermore, as a consequence
\begin{equation} \label{indL2Bound}
    \|\ind_{T,R}^{\text{low}}\|_{L^2_t} \leq \|\ind_T\|_{L^2_t} + \|\ind_{T,R}^{\text{high}}\|_{L^2_t} \les T^{1/2}.
\end{equation}
\end{lem}

\begin{proof}
 See Lemma 2.5 in \cite{MV-15} or Lemma 3.6 in \cite{MolinetUncondKdV}.
\end{proof}

Since we are working with energy estimates, our next concern is our method of shifting derivatives from high to low frequencies. The next lemma essentially allows us to perform this operation.
\begin{lem} \label{commutator}
Let $s \geq 0$, $1 \leq p, q , r \leq \infty$ with $\frac{1}{r} = \frac{1}{p} + \frac{1}{q} $, we have
\begin{equation}
\| [D_x^s P_N , f](g) \|_{L^r(\T)} \lesssim N^{s-1} \|D_xf\|_{L^p(\T)} \|g\|_{L^q(\T)}.
\end{equation}
where $[D_x^sP_N,f]$ is the commutator defined
as
\begin{equation}
[D_x^s P_N,f](g) := D_x^sP_N(fg) - f D_x^sP_Ng.
\end{equation}
\end{lem}
\begin{proof}
See Lemma 2.3 in \cite{FracLeib}.
\end{proof}

Lastly we state some basic estimates which will see varying use.

\begin{lem} \label{QLBounded}
For $1 \leq p \leq \infty$ and $L \in 2^{\N_0}$ it holds
\begin{equation}
\|Q_{\leq L} u \|_{L^p_t L^2_x} \les \|u\|_{L^p_t L^2_x},
\end{equation}
for all $u \in L^p_t L^2_x$. Here the implicit constant is independent of $L$.
\end{lem}

\begin{proof}
   See Lemma 4.1 in \cite{MolinetTanakaUncond}. 
\end{proof}

\begin{lem}[Sobolev Multiplication Law] \label{SobolevMult}
Let $s_1+s_2 \geq 0$, $s_1 \wedge s_2 \geq s_3$ and $s_3 < s_1 + s_2 - 1/2$. Then
\begin{equation} \label{SobolevMultiplication}
\|uv\|_{H^{s_3}} \les \|u\|_{H^{s_1}}\|v\|_{H^{s_2}}.
\end{equation}
Furthermore, for $s > 0$ we have the estimate
\begin{equation} \label{SobolevMultiplicationWeighted}
\|uv\|_{H^s_\omega} \les \|u\|_{H^s_\omega}\|v\|_{L^\infty} + \|u\|_{L^\infty}\|v\|_{H^s_\omega}.
\end{equation}
\end{lem}
\begin{proof}
For \eqref{SobolevMultiplication} see Corollary 3.16 in \cite{TTaoL2}. \eqref{SobolevMultiplicationWeighted} is proven as Lemma 2.2 in \cite{MolinetTanakaUncond}.
\end{proof}

\begin{lem} \label{MVTCounting}
Let $I$ and $J$ be two intervals on the real line and $f \in C^1(J; \R)$. Then
\begin{equation}
   \# \{x \in J \cap \Z : f(x) \in I\} \leq \frac{|I|}{\inf_{x \in J}|f'(x)|} + 1.
\end{equation}
\begin{proof}
See Lemma 2 in \cite{TzvetkozPeriodicKP1}.
\end{proof}
\end{lem}



\section{\texorpdfstring{Failure of $C^3$ regularity of the flow map}{C3 Ill-posedness}} \label{ill-p}

In this section we show that when $\beta>0$, the initial value problem \eqref{MMTEq} cannot be solved via the contraction mapping principle.

\medskip

Let $\Phi_t : H^s(\T) \to H^s(\T)$ denote the solution map associated to \eqref{MMTEq}. 
If one were able to construct solutions by a contraction argument in a Banach space embedded in $C([0,T];H^s)$, then $\Phi_t$ would necessarily be smooth with respect to the initial data; in particular, it would be $C^3$ at the origin.

To test this, consider initial data of the form $\delta u_0(x)$,
with $\delta>0$ small and recall that
\[
d_0^3\Phi_t(u_0,u_0,u_0) = \frac{d^3}{d\delta^3} \Phi_t (\delta u_0) \Big|_{\delta = 0}.
\]
Computing the third Fréchet derivative of $\Phi_t$ at the origin and passing to Fourier variables (see also Section 5 of \cite{PPW-26}), one obtains, up to a nonzero constant $c_\nu$ depending only on $\nu$,
\begin{equation}
\label{eq:thirdFrechet}
c_\nu^{-1}\mathcal F_x [{d_0^3\Phi_t(u_0,u_0,u_0)} ] (k)
=
|k|^{2\beta} e^{it\omega(k)}
\int_0^t
\sum_{k_1-k_2+k_3=k}
e^{i\tau \Omega(\vec{k})}
\widehat u_0(k_1)
\overline{\widehat u_0(k_2)}
\widehat u_0(k_3)
\, d\tau,
\end{equation}
where
\[
\vec{k}=(k_1,k_2,k_3,k),
\qquad
\omega(k)=|k|^\alpha,
\]
and the resonance function is
\[
\Omega(\vec{k})
=
\omega(k)
-
\omega(k_1)
+
\omega(k_2)
-
\omega(k_3).
\]

\medskip

We now choose specific initial data to exhibit the growth of the third derivative \eqref{eq:thirdFrechet}. 
Fix $n\gg1$ and take
\[
\phi_1(x)= e^{ix},
\qquad
\phi_2(x)=a_n e^{inx},
\]
where $a_n\in \mathbb C$ to be determined later. 
Set
\[
u_0=\phi_1+\phi_2 = e^{ix} + a_n e^{inx}.
\]

Consider the high-low-low to high interaction in \eqref{eq:thirdFrechet} corresponding to
\[
k_1=k_2= 1,
\qquad
k_3=k=n.
\]
In this configuration, we have
\[
\Omega(\vec{k})
=
\omega(n)-\omega(1)+\omega(1)-\omega(n) = 0,
\]
so the time integral produces a factor of $t$. 
Therefore,
\[
c_\nu^{-1}\mathcal F_x[{d_0^3\Phi_t(u_0,u_0,u_0)}] (n)
=
n^{2\beta}
e^{it\omega(n)}
t\big(2a_n+|a_n|^2a_n\big).
\]
Here the first term comes from the two high--low--low interactions, while
the second is the all-high interaction.  Taking $a_n>0$, the two
contributions have the same sign.
Using Plancherel’s theorem, we deduce
\[
\|d_0^3\Phi_t(u_0,u_0,u_0)\|_{H^s}^2
\gtrsim
n^{4\beta+2s}
t^2 
|a_n|^2.
\]
We now choose
\[
a_n=n^{-s}, 
\]
so that $\|\phi_2\|_{H^s}\sim1$ and thus $\|u_0\|_{H^s}\sim1$. 
With this normalization,
\[
\|d_0^3\Phi_t(u_0,u_0,u_0)\|_{H^s}
\gtrsim
n^{2\beta} t.
\]

\medskip

Letting $n\to\infty$, 
we see that the third derivative becomes unbounded whenever $\beta>0$. 
Thus the solution map fails to be $C^3$ at the origin for any $\beta>0$. 
Consequently, for $\beta>0$ the IVP \eqref{MMTEq} cannot be solved by the contraction mapping principle. 
To establish well-posedness, we instead rely on a compactness argument combined with \emph{a priori} energy estimates.


\section{Resonance Relation and Strichartz Estimates} \label{sec-4}


The dispersion of \eqref{MMTEq} is exploited in two complementary ways:
either through Strichartz estimates, or via lower bounds on the resonance
relation.  In frequency interactions where the resonance relation does not
yield a sufficiently strong lower bound, we rely instead on Strichartz
estimates.

We begin by recording the resonance structure due to the convexity of the dispersion relation $w (k) = |k|^\alpha$ for $\alpha >1$.

\begin{lem}[Resonance Relation] \label{ResonanceRelation}
Let $1<\alpha\le 2$. 
Under the frequency constraint
\[
\xi_1-\xi_2+\xi_3-\xi_4=0,
\]
the resonance function associated to \eqref{MMTEq} is
\begin{equation}\label{resonance}
\Omega(\xi_1,\xi_2,\xi_3,\xi_4)
:=|\xi_1|^\alpha-|\xi_2|^\alpha+|\xi_3|^\alpha-|\xi_4|^\alpha.
\end{equation}
Moreover, it satisfies the lower bound
\begin{equation}\label{resonance-bound}
|\Omega|
\gtrsim
|\xi_1-\xi_2|\,|\xi_2-\xi_3|\,
|\xi_{\max}|^{\alpha-2},
\end{equation}
where
\[
|\xi_{\max}|:=\max\{|\xi_1|,|\xi_2|,|\xi_3|,|\xi_4|\}.
\]
\end{lem}

\begin{proof}
Recall $\omega (x) = |x|^\alpha$.  
By the mean value theorem,
\begin{equation}\label{ResAppen1}
|\xi_1|^\alpha-|\xi_2|^\alpha
=
(\xi_1-\xi_2)
\int_0^1
\omega'\big(\xi_2+(\xi_1-\xi_2)\theta\big)\,d\theta,
\end{equation}
where $\omega' (\xi) = \alpha |\xi|^{\alpha - 1} \sgn (\xi)$ is continuous everywhere since $\alpha > 1$. Similarly,
\begin{equation}\label{ResAppen2}
\begin{split}
|\xi_3|^\alpha-|\xi_4|^\alpha 
& = |\xi_3|^\alpha-|\xi_1 - \xi_2 + \xi_3|^\alpha \\
& =
(\xi_2-\xi_1)
\int_0^1
\omega'\big(\xi_3+(\xi_1-\xi_2)\theta\big)\,d\theta,
\end{split}
\end{equation}
where we used the constraint $\xi_1-\xi_2+\xi_3-\xi_4=0$, we have
$\xi_4= \xi_1-\xi_2 + \xi_3$.
Hence combining \eqref{ResAppen1} and \eqref{ResAppen2},
\[
\Omega
=
(\xi_1-\xi_2)
\int_0^1
\Big(
\omega'(\xi_2+(\xi_1-\xi_2)\theta)
-
\omega'(\xi_3+(\xi_1-\xi_2)\theta)
\Big)
\,d\theta.
\]

Applying the mean value theorem once more to the difference inside the
integral yields
\[
\omega'(a)-\omega'(b)
=
(a-b)
\int_0^1
\omega''\big(b+\mu(a-b)\big)\,d\mu,
\]
since $\omega' (\xi)$ is continuous everywhere and $\omega ''(\xi) = \alpha (\alpha - 1) |\xi|^{\alpha - 2}$ exists a.e. and is integrable on $[0,1]$ since $\alpha > 1$.
Here, set $a=\xi_2+(\xi_1-\xi_2)\theta$ and
$b=\xi_3+(\xi_1-\xi_2)\theta$.  
A straightforward computation gives
\[
a-b=\xi_2-\xi_3.
\]
Therefore
\[
\Omega
=
(\xi_1-\xi_2)(\xi_2-\xi_3)
\int_0^1\!\!\int_0^1
\omega''\big(\xi_3+(\xi_2-\xi_3)\mu+(\xi_1-\xi_2)\theta\big)
\,d\mu\,d\theta.
\]

Finally, since $\omega''(x)=\alpha(\alpha-1)|x|^{\alpha-2}$ and
$1<\alpha\le 2$, we have
\[
\omega''(x)
\sim
|x|^{\alpha-2}.
\]
On the domain of integration, the argument of $\omega''$, i.e. $ \xi_3+(\xi_2-\xi_3)\mu+(\xi_1-\xi_2) \theta$ is bounded by $O(|\xi_{\max}|)$, which yields
\[
|\Omega|
\gtrsim
|\xi_1-\xi_2|\,|\xi_2-\xi_3|\,
|\xi_{\max}|^{\alpha-2},
\]
since $\alpha \le 2$.
This completes the proof.
\end{proof}

We now consider the relevant Strichartz estimates.
\begin{lem}[Strichartz Estimate] \label{StrichartzPrimary}
Let $1 < \alpha \leq 2$, and let $S(t)$ be the linear propagator defined in \eqref{LinProp}. 
Then for $u \in X^{0,3/8}$, the following Strichartz estimate holds:
\begin{equation} \label{StrichartzTrans}
\|P_{N}u\|_{L^4(\R \times \T)} 
\lesssim 
N^{\frac{2-\alpha}{8}}
\|P_N u\|_{X^{0,3/8}}.
\end{equation}
As a consequence, for any $u_0 \in L^2(\T)$ and any $0 < T \leq 1$, we have
\begin{equation} \label{Strichartz}
\|S(t)P_N u_0\|_{L^4([0,T]\times \T)} 
\lesssim 
N^{\frac{2-\alpha}{8}}
T^{\frac{1}{8}}
\|P_N u_0\|_{L^2(\T)}.
\end{equation}
\end{lem}

When $\alpha \ge 2$, Lemma \ref{StrichartzPrimary} was already established in \cite[Lemma 2.7]{FS-22}, also see \cite[Lemma 4.1]{BLLZ-23}. 
Therefore, we restrict our attention to the case $1 < \alpha \le 2$. Note that when $\alpha < 2$, the Strichartz estimate \eqref{StrichartzTrans} has a derivative loss of $N^{\frac{2-\alpha}{8}}$ comparing the cases $\alpha \ge 2$, which introduces an additional difficulty in the construction of solutions.

A proof can be found in \cite[Corollary 2.11]{ST-2020}, but we provide our own for completeness. The overall strategy of the proof follows closely that of \cite[Lemma 3.1 and Lemma 3.2]{MolinetTanakaUncond}. 
The principal ingredient in our argument is the resonance estimate stated in Lemma \ref{ResonanceRelation}. 

\begin{proof}[Proof of Lemma \ref{StrichartzPrimary}]
We first prove \eqref{StrichartzTrans}. 
The case $\alpha=2$ is classical and was established by Bourgain \cite{Bourgain-93}. 
Thus we focus on the regime $1<\alpha<2$.

By the triangle inequality,
\begin{align}
\begin{split}
\|P_N u\|_{L^4_{tx}}^2
& = \|(P_N u)^2\|_{L^2_{tx}} \\
& = \bigg\|\sum_{L_1,L_2} Q_{L_1}P_N u \; Q_{L_2}P_N u \bigg\|_{L^2_{tx}} \\
&\lesssim
\sum_{L_1\le L_2}
\|Q_{L_1}P_N u \; Q_{L_2}P_N u\|_{L^2_{tx}}.
\end{split}
\label{StrichartzPf1}
\end{align}
We claim that for $L_1\le L_2$,
\begin{equation}\label{StrichartzPffSuff}
\|Q_{L_1}P_N u \; Q_{L_2}P_N u\|_{L^2_{tx}}
\lesssim
N^{\frac{2-\alpha}{4}}
L_1^{1/2}L_2^{1/4}
\|Q_{L_1}P_N u\|_{L^2_{tx}}
\|Q_{L_2}P_N u\|_{L^2_{tx}}.
\end{equation}
Assuming \eqref{StrichartzPffSuff}, 
we now complete the proof of \eqref{StrichartzTrans}.
Write $L_2=2^kL_1$ with $k\ge0$. Then, by \eqref{StrichartzPffSuff}, we have
\begin{align*}
\eqref{StrichartzPf1}
&\lesssim
N^{\frac{2-\alpha}{4}}
\sum_{k\ge0}
\sum_{L_1}
L_1^{1/2}(2^kL_1)^{1/4}
\|Q_{L_1}P_N u\|_{L^2}
\|Q_{2^kL_1}P_N u\|_{L^2} \\
&=
N^{\frac{2-\alpha}{4}}
\sum_{k\ge0}
2^{-k/8}
\sum_{L_1}
L_1^{3/8}\|Q_{L_1}P_N u\|_{L^2}
(2^kL_1)^{3/8}\|Q_{2^kL_1}P_N u\|_{L^2}.
\end{align*}
Applying Cauchy–Schwarz in $L_1$,
\[
\eqref{StrichartzPf1}
\lesssim
N^{\frac{2-\alpha}{4}}
\sum_{k\ge0}
2^{-k/8}
\Big(
\sum_{L_1} L_1^{3/4}\|Q_{L_1}P_N u\|_{L^2}^2
\Big)^{1/2}
\Big(
\sum_{L_1}(2^kL_1)^{3/4}\|Q_{2^kL_1}P_N u\|_{L^2}^2
\Big)^{1/2}.
\]
Since $\sum_{k\ge0}2^{-k/8}<\infty$, this yields
\[
\|P_N u\|_{L^4_{tx}}^2
\lesssim
N^{\frac{2-\alpha}{4}}
\|P_N u\|_{X^{0,3/8}}^2,
\]
which completes the reduction. 
Thus, for \eqref{StrichartzTrans}, 
it remains to prove \eqref{StrichartzPffSuff}.

\medskip

By Plancherel and Cauchy–Schwarz in $(n_1,\tau_1)$,
\begin{align*}
\|Q_{L_1}P_N u \; Q_{L_2}P_N u\|_{L^2_{tx}}^2
&=
\int_\tau \sum_n
\bigg|
\int_{\tau_1}\sum_{n_1}
\widehat{Q_{L_1}P_N u}(n_1,\tau_1)
\widehat{Q_{L_2}P_N u}(n-n_1,\tau-\tau_1)
\bigg|^2 \\
&\lesssim
\sup_{\tau,n} A(\tau,n)\,
\|Q_{L_1}P_N u\|_{L^2_{tx}}^2
\|Q_{L_2}P_N u\|_{L^2_{tx}}^2,
\end{align*}
where
\[
A(\tau,n)
:=
\sum_{n_1\in\Z}\int_{\R}
\mathbbm{1}_{\supp\widehat{Q_{L_1}P_N u}}(n_1,\tau_1)
\mathbbm{1}_{\supp\widehat{Q_{L_2}P_N u}}
(n-n_1,\tau-\tau_1)\,d\tau_1.
\]

Hence it suffices to show
\[
\sup_{\tau,n} A(\tau,n)
\lesssim
N^{\frac{2-\alpha}{2}}L_1L_2^{1/2}.
\]
Using the frequency and modulation localization,
\begin{align*}
A(\tau,n)
&\lesssim
L_1
\#\Big\{
n_1:
|n_1|\sim|n-n_1|\sim N,\;
|\tau-|n_1|^\alpha-|n-n_1|^\alpha|
\lesssim L_2
\Big\} \\
&=: L_1\,\# B(\tau,n).
\end{align*}
To bound $\# B(\tau,n)$, fix $R>0$ (to be chosen) and split
\[
B(\tau,n)
=
\Big( B(\tau,n) \cap \{ |n-2n_1|\le R\}\Big)
\cup \Big( B(\tau,n) \cap 
\{ |n-2n_1|>R\} \Big).
\]
The first set contains at most $R+1$ elements:
\begin{equation}\label{CountingBound1}
\#\{ |n-2n_1|\le R\}\lesssim R+1.
\end{equation}
For the second set, define
\[
f(n_1)=\tau-|n_1|^\alpha-|n-n_1|^\alpha.
\]
Then
\[
\begin{split}
|f'(n_1)|
& =
\alpha\big|
(n-n_1)|n-n_1|^{\alpha-2}
-
n_1|n_1|^{\alpha-2}
\big|\\
& = (\alpha - 1) |n - 2 n_1| \int_0^1 |n_1 + s (n-2n_1)|^{\alpha - 2} ds \\
& \gtrsim |n-2n_1|\,N^{\alpha-2} \gtrsim
R\,N^{\alpha-2},
\end{split}
\]
where we used $|n_1|\sim|n-n_1|\sim N$, $1<\alpha \le 2$, and $|n-2n_1|>R$.
The set $|n-2n_1|>R$ is the union of at most two integer intervals. Applying
Lemma~\ref{MVTCounting} on each interval,
\begin{equation}\label{CountingBound2}
\# \Big( B(\tau,n) \cap 
\{ |n-2n_1|>R\} \Big)
\lesssim
\frac{L_2}{R N^{\alpha-2}}+1.
\end{equation}
Choosing $R\sim (L_2 N^{2-\alpha})^{1/2}$ balances
\eqref{CountingBound1} and \eqref{CountingBound2}, yielding
\[
\#B (\tau,n)
\lesssim
N^{\frac{2-\alpha}{2}}L_2^{1/2}.
\]
Therefore
\[
A(\tau,n)
\lesssim
N^{\frac{2-\alpha}{2}}L_1L_2^{1/2},
\]
which proves \eqref{StrichartzPffSuff} and hence
\eqref{StrichartzTrans}.

\medskip

Finally, to prove \eqref{Strichartz}, let $\psi\in C_c^\infty(\R)$ satisfy
$\psi\equiv1$ on $[-1,1]$ and $\psi=0$ outside $[-2,2]$.
Set $\psi_T(t)=\psi(t/T)$. Then
\begin{align*}
\|S(t)P_N u_0\|_{L^4([0,T]\times\T)}
&\le
\|\psi_T S(t)P_N u_0\|_{L^4_{tx}} \\
&\lesssim
N^{\frac{2-\alpha}{8}}
\|\psi_T P_N u_0\|_{H^{3/8}_tL^2_x} \\
&=
N^{\frac{2-\alpha}{8}}
\|\psi_T\|_{H^{3/8}_t}
\|P_N u_0\|_{L^2_x}.
\end{align*}
Since $\|\psi_T\|_{H^{3/8}_t}\lesssim T^{1/8}\|\psi\|_{H^{3/8}}$,
the desired estimate follows.
\end{proof}

We use the following values to simplify later calculations.
\begin{definition}
\label{DEF:kg}
For $1 < \alpha \leq 2$ and $\beta\geq0$, we define
\begin{equation}
\kappa(\alpha,\beta):= \frac{\beta}{4} + \frac{2-\alpha}{8}, \hspace{1cm} \gamma(\alpha,\beta) := \beta + 2\kappa(\alpha,\beta) = \frac{3}{2}\beta + \frac{2-\alpha}{4}.
\end{equation}
\end{definition}

The manner in which we exploit the Strichartz estimates is summarized in the following corollary.

\begin{cor}\label{StrichartzSecondary}
Let $s > \frac12$ and $0<T\le 1$. 
Suppose $u \in L^\infty([0,T];H^s_\omega(\T))$ is a distributional
solution to \eqref{MMTEq} with initial data $u_0 \in H^s(\T)$ on $[0,T]$. 
Then the following estimate holds:
\begin{equation}\label{ell4L4Bound}
\left(
\sum_N 
 \omega_N^4\|D_x^{\,s-\kappa(\alpha,\beta)} P_N u\|_{L^4([0,T]\times\T)}^4
\right)^{1/4}
\lesssim 
T^{\frac18}
\big(1 + T\|u\|_{L^\infty_T H^s}^2\big)
\|u\|_{L^\infty_T H^s_\omega}.
\end{equation}
\end{cor}

\begin{proof}
It suffices to consider $N\ge1$. We partition the time interval into subintervals
$\{I_{j,N}\}_{j\in J_N}$ such that
\[
\bigcup_{j\in J_N} I_{j,N} = [0,T],
\qquad
|I_{j,N}|\sim T N^{-\theta},
\qquad
|J_N|\lesssim N^\theta,
\]
where $\theta>0$ will be chosen later.

For each fixed $N$, the projected equation and $s>1/2$ imply
$P_Nu\in W^{1,\infty}([0,T];L^2(\T))$: indeed, after applying $P_N$, all
spatial multipliers are bounded and $|u|^2u\in L^\infty_TH^s$.  Thus
$t\mapsto \|D_x^s P_N u(t)\|_{L^2(\T)}$ is continuous on $[0,T]$.
for each $I_{j,N}$ we may choose $c_{j,N}\in I_{j,N}$ such that
$\|D_x^s P_N u(t)\|_{L^2}^4$ attains its minimum at $t=c_{j,N}$. 
In particular,
\begin{equation}\label{cjProperty}
|I_{j,N}|\,\|D_x^s P_N u(c_{j,N})\|_{L^2_x}^4
\le
\int_{I_{j,N}}
\|D_x^s P_N u(t)\|_{L^2_x}^4\,dt.
\end{equation}
By the Duhamel formula \eqref{Duhamel}, for $t\in I_{j,N}$,
\begin{equation}\label{XP-1}
P_N u(t)
=
S(t-c_{j,N})P_N u(c_{j,N})
+
\int_{c_{j,N}}^t
S(t-t')P_N D_x^{2\beta}(|u|^2u)(t')\,dt'.
\end{equation}
We estimate the two terms separately.

\medskip

\noindent
For the linear term,
using \eqref{Strichartz}, Bernstein’s inequality, and \eqref{cjProperty},
\begin{align*}
&\Big(
\sum_N \sum_{j\in J_N}
\omega_N^4\|D_x^{s-\kappa(\alpha,\beta)}
S(t-c_{j,N})P_N u(c_{j,N})\|_{L^4(I_{j,N};L^4_x)}^4
\Big)^{1/4} \\
&\qquad
\lesssim
\Big(
\sum_N \sum_{j\in J_N}
\omega_N^4|I_{j,N}|^{1/2}
N^{-\beta}
\|D_x^s P_N u(c_{j,N})\|_{L^2_x}^4
\Big)^{1/4} \\
&\qquad
\lesssim
\Big(
\sum_N \sum_{j\in J_N}
\omega_N^4T^{-1/2}
N^{\theta/2-\beta}
|I_{j,N}|
\|D_x^s P_N u(c_{j,N})\|_{L^2_x}^4
\Big)^{1/4} \\
&\qquad
\lesssim
\Big(
\sum_N
T^{-1/2}
N^{\theta/2-\beta}
\int_0^T
\omega_N^4\|D_x^s P_N u(t)\|_{L^2_x}^4 dt
\Big)^{1/4}.
\end{align*}
Choosing $\theta=2\beta$ and using the embedding
$\ell^2(\N)\hookrightarrow \ell^4(\N)$, we obtain
\[
\lesssim
T^{1/8}\|u\|_{L^\infty_T H^s_\omega}.
\]

\medskip

\noindent
Now, we turn to the Duhamel term.
We next consider
\begin{align*}
&\Big(
\sum_N \sum_{j\in J_N}
\omega_N^4\Big\|
\int_{c_{j,N}}^t
S(t-t')
D_x^{s+2\beta-\kappa(\alpha,\beta)}
P_N(|u|^2u)(t')
\,dt'
\Big\|_{L^4(I_{j,N};L^4_x)}^4
\Big)^{1/4}.
\end{align*}
Using the Christ–Kiselev lemma together with \eqref{Strichartz},
\begin{align*}
\lesssim
\Big(
\sum_N \sum_{j\in J_N}
\omega_N^4|I_{j,N}|^{1/2}
\Big(
\int_{I_{j,N}}
\|D_x^{s+2\beta-\beta/4}
P_N(|u|^2u)(t')\|_{L^2_x}
dt'
\Big)^4
\Big)^{1/4}.
\end{align*}
Applying Bernstein and Hölder,
\begin{align*}
\lesssim
\Big(
\sum_N \sum_{j\in J_N}
N^{7\beta}
|I_{j,N}|^{7/2}
\int_{I_{j,N}}
\omega_N^4\|D_x^s P_N(|u|^2u)(t')\|_{L^2_x}^4 dt'
\Big)^{1/4}.
\end{align*}
Since $|I_{j,N}|\sim T N^{-\theta}$ and $\theta=2\beta$, we deduce
\[
\lesssim
T^{7/8}
\Big(
\sum_{N\ge1} 
\int_0^T
\omega_N^4\|D_x^s P_N(|u|^2u)(t')\|_{L^2_x}^4 dt'
\Big)^{1/4}.
\]
With $\theta=2\beta$, the frequency factor cancels, and using again
$\ell^2(\N)\hookrightarrow\ell^4(\N)$,
\[
\lesssim
T^{9/8}
\||u|^2u\|_{L^\infty_T H^s_\omega}.
\]
Making use of \eqref{SobolevMultiplicationWeighted} and the embedding
$H^s_x \hookrightarrow L^\infty_x$ for $s > 1/2$, the Duhamel contribution is bounded by
\[
\lesssim
T^{9/8}\|u\|_{L^\infty_T H^s}^2\|u\|_{L^\infty_T H^s_\omega}.
\]

\medskip

Combining the linear and nonlinear estimates yields
\eqref{ell4L4Bound}, completing the proof.
\end{proof}

We also obtain an analogous estimate for the difference of two solutions.

\begin{cor} \label{Diffell4L4Bound}
Let $s > \frac12$ and $0<T\le 1$. 
Suppose $u,v \in L^\infty([0,T];H^s(\T))$ are distributional solutions to
\eqref{MMTEq} with initial data $u_0,v_0 \in H^s(\T)$. 
Let $w = u - v$. 
Then the following estimate holds:
\begin{equation} \label{Diffell4L4BoundEq}
\left(
\sum_N 
\|D_x^{\,s-1-\kappa(\alpha,\beta)} P_N w\|_{L^4([0,T]\times\T)}^4
\right)^{1/4}
\lesssim
T^{\frac18}
\big(1 + T(\|u\|_{L^\infty_T H^s}^2 + \|v\|_{L^\infty_T H^s}^2)\big)
\|w\|_{L^\infty_T H^{s-1}}.
\end{equation}
\end{cor}

\begin{proof}
Apply the proof of Corollary~\ref{StrichartzSecondary} to the equation for
$w=u-v$. The linear contribution is bounded by
$T^{1/8}\|w\|_{L^\infty_TH^{s-1}}$. For the Duhamel contribution, the same
short-time decomposition reduces the estimate to
\[
T^{9/8}\big\||u|^2u-|v|^2v\big\|_{L^\infty_TH^{s-1}}.
\]
Since $s>1/2$, repeated use of \eqref{SobolevMultiplication} gives
\[
\big\||u|^2u-|v|^2v\big\|_{H^{s-1}}
\lesssim
\big(\|u\|_{H^s}^2+\|v\|_{H^s}^2\big)\|u-v\|_{H^{s-1}}.
\]
Combining these bounds proves \eqref{Diffell4L4BoundEq}.
\end{proof}

\section{Estimates on Solutions}\label{sec-5}


In this section, we establish several estimates for solutions to the IVP \eqref{MMTEq}, including an {\em a priori estimate} that plays a central role in our analysis. We begin with the following lemma.

\begin{lem}\label{BourgainEstimates}
Let $0<T\le1$ and $s>\frac12$. 
Suppose $u \in L^\infty_T H^s_\omega$ is a solution to \eqref{MMTEq} with initial data $u_0 \in H^s$. 
Then $u$ belongs to
\[
Z_{\omega,T}^s := L^\infty_T H^s_\omega \cap X_{\omega,T}^{s-2\beta,1},
\]
and satisfies the bound
\begin{equation}\label{BoundOnBourgain}
\|u\|_{Z_{\omega,T}^s}
\lesssim
\|u\|_{L^\infty_T H^s_\omega}
\big(1+\|u\|_{L^\infty_T H^s}^2\big).
\end{equation}
Moreover, if $u$ and $v$ are solutions with respective initial data $u_0$ and $v_0$, then
\begin{equation}\label{BoundOnBourgainDiff}
\|u-v\|_{Z_T^{s-1}}
\lesssim
\|u-v\|_{L^\infty_T H^{s-1}}
\big(1+(\|u\|_{L^\infty_T H^s}+\|v\|_{L^\infty_T H^s})^2\big).
\end{equation}
\end{lem}

\begin{proof}
It is sufficient to bound only the $X^{s-2\beta,1}_{\omega,T}$ norm. Since $u$ satisfies the Duhamel formulation \eqref{Duhamel}, by Lemma \ref{StandardLinear} and Lemma \ref{SobolevMult}, we get
\begin{equation}
\begin{split}
\|u\|_{X^{s-2\beta,1}_{\omega,T}} &\lesssim \|u_0\|_{H^{s-2\beta}_\omega} + \|D_x^{2\beta}(|u|^2u)\|_{L^2_T H^{s-2\beta}_\omega} \\
& \lesssim\|u\|_{L^\infty_T H^s_\omega} + \|u\|_{L^\infty_{T}H^s}^2\|u\|_{L^\infty_T H^s_\omega},
\end{split}
\end{equation}
which gives \eqref{BoundOnBourgain}.
The estimate \eqref{BoundOnBourgainDiff} follows by the same argument and the use of telescopic sums.
\end{proof}

We are now in a position to derive {\it a priori} bounds for solutions to \eqref{MMTEq} via energy estimates.

\begin{prop}[A priori estimate]\label{Apriori}
Let $0<T\le1$, $0<\beta\le \alpha/4$, and 
\[
s \ge \max\big(\gamma(\alpha,\beta),\, \tfrac12+\big),
\]
where $\gamma (\alpha,\beta)$ is defined in Definition \ref{DEF:kg}.
Suppose $u \in L^\infty([0,T];H^s_\omega(\T))$ is a solution to \eqref{MMTEq}. 
Then there exist constants $\theta > 0$ and $C>0$, depending on
$\alpha,\beta$, and $s$, such that
\begin{equation}
\|u\|_{L^\infty_T H^s_\omega}^2
\le
\|u_0\|_{H^s_\omega}^2
+
C T^\theta
\|u\|_{L^\infty_T H^s_\omega}
\|u\|_{Z_{\omega,T}^s} \|u\|_{L^\infty_T H^s} \|u\|_{Z^s_T}
\big(1+\|u\|_{L^\infty_T H^s}^2\big)^4.
\end{equation}
\end{prop}

\begin{proof}
For the proof, use the equivalent dyadic energy
\[
\mathcal E^s_\omega(u)
:=
|\widehat u(0)|^2+
\sum_{N\ge1}\omega_N^2\|D_x^sP_Nu\|_{L^2}^2
\sim \|u\|_{H^s_\omega}^2.
\]
For any solution to \eqref{MMTEq} and every dyadic $N\ge1$, one has
\begin{equation}
\begin{split}  
   \frac{d}{dt}\left(\omega_N^2\|D_x^sP_Nu(t)\|_{L^2}^2\right) &  =  2\nu \omega_N^2\Im \Big[\int_\T D_x^{2\beta + 2s} P_N^2 \bar{u}  \cdot |u|^2 u \, dx \Big],
\end{split}
\end{equation}
which implies that
\begin{equation}
\begin{split} \label{IntegralOvert}
\omega_N^2\left(\|D_x^sP_Nu(T)\|_{L^2}^2-\|D_x^sP_Nu_0\|_{L^2}^2\right) = 2\nu \omega_N^2\Im \Big[\int_0^T \int_\T  D_x^{2\beta + 2s} P_N^2\bar{u} \cdot |u|^2u\, dx dt \Big] .
\end{split}
\end{equation}
Due to the symmetry, we can rewrite the integral term in \eqref{IntegralOvert} as follows,
\begin{equation}
\begin{split} \label{PMCommutator}
2\Im & \Big[ \int_\T D_x^{2\beta+2s} P_N^2 \bar{u} \cdot |u|^2u \, dx\Big] \\
&= \frac{1}{2} \Im\Big[\sum_{k_1 -k_2 +k_3 - k_4 = 0} A(k_1,k_2,k_3,k_4) \bar{\hat{u}}(k_4) \hat{u}(k_3) \bar{\hat{u}}(k_2) \hat{u}(k_1)  \Big],
\end{split}
\end{equation}
where
\begin{equation} \label{symbolsApriori}
A(k_1,k_2,k_3,k_4):=m(k_4)-m(k_3)+m(k_2)-m(k_1),
\end{equation}
with $m(k) = \phi_N(k)^2|k|^{2\beta+2s}$.  
In what follows, 
we estimate the right hand side of \eqref{IntegralOvert}.
For $N \les 1$, we see that
\[
\bigg|  \int_0^T \int_\T D_x^{2\beta+2s} P_N^2 \bar{u} \cdot |u|^2u \, dx dt \bigg| \les T \|u\|_{L^\infty([0,T]\times \T)}^4 \les T \|u\|_{L^\infty_{T}([0,T]; H^{\frac12+}(\T))}^4,
\]
which is sufficient for our purpose.
Hence, we may assume that $N \gg 1$ in the following.

We begin by dyadically decomposing each term in \eqref{PMCommutator} into frequency scales $N_1,N_2,N_3,N_4$. 
For brevity, denote the remaining $N \gg 1$ part of \eqref{IntegralOvert} by ${\rm J}_T$. 
We then obtain
\begin{align}
\label{JT}
\begin{split}
{\rm J}_T
&= \sum_{N\gg1}\omega_N^2
\left(\|D_x^sP_Nu(T)\|_{L^2}^2-\|D_x^sP_Nu_0\|_{L^2}^2\right) \\
&= 2\nu \sum_{N \gg 1} \omega_N^2\Im \bigg[
\int_0^T \int_\T 
D_x^{2\beta + 2s} P_N^2 \bar{u} \cdot |u|^2 u 
\, dx\, dt
\bigg] \\
&= \frac{\nu}{2}
\sum_{N \gg 1}\omega_N^2
\Im \bigg[
\int_0^T 
\sum_{k_1 - k_2 + k_3 - k_4 = 0}
A(k_1,k_2,k_3,k_4)\,
\overline{\hat{u}}(k_4)
\hat{u}(k_3)
\overline{\hat{u}}(k_2)
\hat{u}(k_1)
\, dt
\bigg] \\
&\leq \frac{1}{2}
\sum_{\substack{N \gg 1 \\ N_1,N_2,N_3,N_4}}\omega_N^2
\bigg|
\int_0^T
\sum_{k_1 - k_2 + k_3 - k_4 = 0}
A(k_1,k_2,k_3,k_4)\\
& \hphantom{XXXXXXXXXX} \times
\overline{\widehat{P_{N_4} u}}(k_4)
\widehat{P_{N_3} u}(k_3)
\overline{\widehat{P_{N_2} u}}(k_2)
\widehat{P_{N_1} u}(k_1)
\, dt
\bigg|.
\end{split}
\end{align}

By symmetry, to estimate \eqref{JT} we can assume the following reductions:
$N_1 \geq N_2,N_3$ and $N_2 \geq N_4$. Furthermore, if $k_1 = k_2$ or
$k_2 = k_3$, then the summand is real and does not contribute to the
imaginary part. We may therefore assume that the resonance function in
Lemma \ref{ResonanceRelation} is nonzero.
Since $k_1 -k_2 + k_3 - k_4 = 0$, the two largest frequencies must be comparable, that is $N_1 \sim N_2$ or $N_1 \sim N_3$. We then distinguish cases based on frequency interaction which give rise to different lower bounds on the resonance relation.
\begin{itemize}
    \item {\it Case 1}: $N_1 \sim N_2 \sim N_3 \sim N_4$
    \item {\it Case 2}: $N_1 \sim N_3 \gg  N_4$
    \item {\it Case 3}: $N_1 \sim N_2 \gg  N_4$
\end{itemize}

Let ${\rm J}_T^{(i)}$ denote ${\rm J}_T$ defined in \eqref{JT} restricted to the $i$-th case. Let us denote by $N_{(1)} \geq N_{(2)} \geq N_{(3)} \geq N_{(4)}$ an order on $N_1,N_2,N_3,N_4$. By the reductions from before $N_1 \sim N_{(1)}$. Now, we move to estimate each case in detail.\\ 

\noindent
{\it \large Case 1: $N_1 \sim N_2 \sim N_3 \sim N_4$}.
In this case, by H\"older and Corollary \ref{StrichartzSecondary}.
\begin{equation}
\begin{split}
{\rm J}_T^{(1)} &\les \sum_{N_{1}\sim N_2 \sim N_3 \sim N_4} \omega_{N_1}^2 N_1^{2s + 2\beta}\|P_{N_1}u\|_{L^4_{Tx}}\|P_{N_2}u\|_{L^4_{Tx}}\|P_{N_3}u\|_{L^4_{Tx}}\|P_{N_4}u\|_{L^4_{Tx}} \\
& \les \Big(\sum_N \omega_N^4\|D_x^{s - \kappa(\alpha,\beta)}P_N u\|_{L^4_{Tx}}^4\Big)^{1/2} \times \Big(\sum_N \|D_x^{\beta + \kappa(\alpha,\beta)}P_N u\|_{L^4_{Tx}}^4\Big)^{1/2} \\
& \les T^{1/2}\|u\|_{L^\infty_T H^s_\omega}^2 \|u\|_{L^\infty_T H^{\gamma(\alpha,\beta)}}^2\big(1 + T\|u\|_{L^\infty_T H^{s}}^2\big)^4,
\end{split}
\end{equation}
since $s \ge \max (\gamma(\alpha,\beta),\, \tfrac12+)$,
which is acceptable.\\

\noindent
{\it \large Case 2: $N_1 \sim N_3 \gg  N_4$}.
In this case, by Lemma \ref{ResonanceRelation} the resonance relation satisfies $|\Omega| \gtrsim N_{(1)}^\alpha =: L$.

Observe that we may bound $J_T^{(2)}$ simply by
\begin{equation}
\begin{split} \label{J2}
{\rm J}_T^{(2)}& \lesssim  \sum_{
N_1 \sim N_3 \gg  N_4} \omega_{N_4}^2\left|\int_0^T \int_\T D_x^{2\beta + 2s}P_N^2 P_{N_4}\bar{u} P_{N_3}u P_{N_2}\bar{u} P_{N_1}u \, dx dt'\right| \\ 
&\quad +\sum_{ 
N_1 \sim N_3 \gg   N_4} \omega_{N_3}^2 \left|\int_0^T \int_\T P_{N_4}\bar{u}  D_x^{2\beta + 2s}P_N^2  P_{N_3}u P_{N_2}\bar{u} P_{N_1}u \, dx dt'\right| \\ 
&\quad+\sum_{
N_1 \sim N_3 \gg    N_4} \omega_{N_2}^2 \left|\int_0^T \int_\T P_{N_4}\bar{u}    P_{N_3}u D_x^{2\beta + 2s}P_N^2P_{N_2}\bar{u} P_{N_1}u \, dx dt'\right| \\ 
&\quad+\sum_{
N_1 \sim N_3 \gg   N_4}  \omega_{N_1}^2\left|\int_0^T \int_\T P_{N_4}\bar{u}    P_{N_3}u P_{N_2}\bar{u} D_x^{2\beta + 2s}P_N^2P_{N_1}u \, dx dt'\right|.
\end{split}
\end{equation}
Since $s,\beta\ge0$, it suffices to estimate the last term, where $D_x^{2\beta + 2s}$ hits on the highest frequency, as the others are treated similarly. We argue by showing the multilinear estimate
\begin{equation} \label{multline1}
\begin{split}
    \omega_{N_{(1)}}^2\bigg| \int_0^T \int_{\T} P_{N_1}v_1 P_{N_2} \bar{v}_2 P_{N_3} v_3 P_{N_4} \bar{v}_4 \, dx dt \bigg| \\
    & \hspace{-6cm}\les T^\theta \omega_{N_{(1){}}}\omega_{N_{(2)}} N_{(1)}^{-\alpha}(N_{(3)}N_{(4)})^\frac{1}{2} \sum_{i=1}^4 \left(N_{(1)}^{0+}\|P_{N_i} v_i \|_{L^\infty_T L^2_x} + \|P_{N_i} v_i \|_{X^{0,1}_T} \right) \prod_{\substack{j = 1 \\ j \neq i}}^4 \|P_{N_j} v_j\|_{L^\infty_T L^2_x},
\end{split}
\end{equation}
where $\theta$ depends only on $\alpha$.
Assuming \eqref{multline1}, then
\begin{equation}
\begin{split}
{\rm J}_T^{(2)} &\les T^\theta \sum_{N_1 \sim N_3 \gg N_4} \omega_{N_1}\omega_{N_2}N_{(1)}^{2(\beta + s) - \alpha}N_2^\frac12 N_4^\frac12 \\
& \qquad \qquad \times \sum_{i=1}^4 \left( N_{(1)}^{0+} N_{(i)}^{-s}\|P_{N_{(i)}} u\|_{L^\infty_T H^s} + N_{(i)}^{2\beta-s}\|P_{N_{(i)}} u \|_{X^{s-2\beta,1}_T}\right)\prod_{\substack{j = 1 \\ j\neq i}}^4 N_{(j)}^{-s}\|P_{N_{(j)}} u \|_{L^\infty_T H^s}.
\end{split}
\end{equation}
Collecting coefficients, the two dyadic factors that must be summed are
\begin{equation} \label{boundConditions}
N_{(1)}^{2\beta-\alpha+}N_2^{\frac12-s}N_4^{\frac12-s},
\qquad
N_{(1)}^{2\beta-\alpha}N_2^{\frac12-s}N_4^{\frac12-s}
N_{(i)}^{2\beta}.
\end{equation}
The sums in $N_2$ and $N_4$ are bounded because $s>1/2$. The first factor
is summable since $2\beta<\alpha$. In the second factor the worst case is
$N_{(i)}\sim N_{(1)}$, which gives $N_{(1)}^{4\beta-\alpha}$ and is bounded
under $4\beta\le\alpha$.

It remains to show \eqref{multline1}. 
We pass to extensions to the real line $\rho_T(v)$ as in \eqref{extension}, which by a slight abuse of notation we still denote an extension of $v$ by itself. Introducing indicator functions to ensure the time domain of the integral remains consistent, we can denote the LHS of \eqref{multline1} by
\begin{equation}
\begin{split}
&\omega_{N_{(1)}}^2 \bigg|\int_0^T \int_\T P_{N_1} \ind_T v_1  P_{N_2} \bar v_2 P_{N_3} \ind_T v_3 P_{N_4} \bar v_4 \, dx dt \bigg| \\
& \qquad \qquad =: {\rm I}^{(2)}(P_{N_1}\ind_T v_1 , P_{N_2} v_2, P_{N_3}\ind_T v_3, P_{N_4}v_4 )
\end{split}
\end{equation}
We decompose ${\rm I}^{(2)}$ using \eqref{indDecomp} according to whether the time cutoff lies in the high- or low-frequency region,
\begin{equation}
\begin{split}
{\rm I}^{(2)} & (P_{N_1}\ind_T v_1, P_{N_2}v_2, P_{N_3}\ind_T v_3, P_{N_4} v_4  ) \\
& \quad \leq {\rm I}^{(2)}(P_{N_1}\ind_{T,R}^\text{high} v_1 , P_{N_2}v_2, P_{N_3}\ind_T v_3, P_{N_4}v_4  ) \\
&\quad + {\rm I}^{(2)}(P_{N_1}\ind_{T,R}^\text{low} v_1, P_{N_2}v_2, P_{N_3}\ind_{T,R}^\text{high} v_3, P_{N_4}v_4  ) \\
&\quad + {\rm I}^{(2)}(P_{N_1}\ind_{T,R}^\text{low}v_1, P_{N_2}v_2, P_{N_3}\ind_{T,R}^\text{low} v_3, P_{N_4}v_4  ) \\
& =: {\rm I}_1^{(2)} + {\rm I}_2^{(2)} + {\rm I}_3^{(2)}.
\end{split}    
\end{equation}
With the choice of $R = N_{(1)}^{\alpha -}$ and using \eqref{indL1Bound} one has,
\begin{equation}
\begin{split}
{\rm I}_1^{(2)} &\les \omega_{N_1}^2 \|\indh\|_{L^1_t}\|P_{N_1}v_1\|_{L^\infty_t L^2_x}\|P_{N_3}v_3\|_{L^\infty_t L^2_x}\|P_{N_2}v_2\|_{L^\infty_{tx}}\|P_{N_4}v_4\|_{L^\infty_{tx}} \\
& \les \bigg(\frac{\omega_{N_1}}{\omega_{N_3}}\bigg) \omega_{N_1}\omega_{N_3} T^\theta N_1^{-\alpha +}N_2^\frac12 N_4^\frac12 \|P_{N_1}v_1\|_{L^\infty_t L^2_x}\|P_{N_3}v_3\|_{L^\infty_t L^2_x}\|P_{N_2}v_2\|_{L^\infty_{t}L^2_x}\|P_{N_4}v_4\|_{L^\infty_{t}L^2_x},
\end{split}
\end{equation}
which is sufficient for our purposes since $\omega_{N_1}/\omega_{N_3} \les 1$. The same argument provides the same bound for ${\rm I}_2^{(2)}$.
\begin{remark} \label{freqEnvolopeRatioBound}
Since the frequency envelope satisfies $\omega_N \leq \omega_{2N} \leq \delta \omega_{N}$, given there exists $C > 0$ such that $N_1 \leq CN_3$, then 
\begin{equation}
\frac{\omega_{N_1}}{\omega_{N_3}} \les \delta^{\log_2(C)} \les C
\end{equation}
uniformly in $\delta$, since $\delta  \leq 2$. It is important that the bound is uniform in $\delta$ since later $\delta$ will depend on initial data. See Lemma 4.6 of \cite{MolinetTanakaUncond}.
\end{remark}
We now turn to the estimate of $I_3^{(2)}$. 
We further decompose according to modulation as follows:
\begin{equation}
\begin{split}
{\rm I}_3^{(2)}
&\le
{\rm I}^{(2)}(P_{N_1}Q_{\gtrsim L}(\indl v_1), P_{N_2}v_2, P_{N_3}\indl v_3, P_{N_4}v_4) \\
&\quad+
{\rm I}^{(2)}(P_{N_1}Q_{\ll L}(\indl v_1), P_{N_2}Q_{\gtrsim L}v_2, P_{N_3}\indl v_3, P_{N_4}v_4) \\
&\quad+
{\rm I}^{(2)}(P_{N_1}Q_{\ll L}(\indl v_1), P_{N_2}Q_{\ll L} v_2, P_{N_3}Q_{\gtrsim L}(\indl v_3), P_{N_4}v_4) \\
&\quad+
{\rm I}^{(2)}(P_{N_1}Q_{\ll L}(\indl v_1), P_{N_2}Q_{\ll L}v_2, P_{N_3}Q_{\ll L}(\indl v_3), P_{N_4}Q_{\gtrsim L}v_4) \\
&=: {\rm I}_{3,1}^{(2)} + {\rm I}_{3,2}^{(2)} + {\rm I}_{3,3}^{(2)} + {\rm I}_{3,4}^{(2)}.
\end{split}
\end{equation}
We estimate each term separately. Using \eqref{indLinftyBound}, \eqref{indL2Bound}, the fact that 
$L = N_{(1)}^{\alpha} \gg R = N_{(1)}^{\alpha -}$, Lemma~\ref{QLlow} and Remark \ref{freqEnvolopeRatioBound} we obtain
\begin{equation}
\begin{split}
{\rm I}_{3,1}^{(2)}
&\lesssim
\omega_{N_1}^2
\|\indl\|_{L^2_t}
\|P_{N_1}Q_{\gtrsim L}(\indl v_1)\|_{L^2_{tx}}
\|P_{N_3}v_3\|_{L^\infty_t L^2_x}
\|P_{N_2}v_2\|_{L^\infty_{tx}}
\|P_{N_4}v_4\|_{L^\infty_{tx}} \\
&\lesssim
\omega_{N_1}\omega_{N_3}T^{1/2}
N_{(1)}^{-\alpha}N_2^\frac12 N_4^\frac12
\|P_{N_1}v_1\|_{X^{0,1}}
\|P_{N_3}v_3\|_{L^\infty_t L^2_x} 
\|P_{N_2}v_2\|_{L^\infty_t L^2_x}
\|P_{N_4}v_4\|_{L^\infty_t L^2_x}. 
\end{split}
\end{equation}
The same argument gives a similar bound for ${\rm I}^{(2)}_{3,3}$. Furthermore, by similar reasoning
\begin{equation}
\begin{split}
{\rm I}^{(2)}_{3,2} & \les \omega_{N_1}^2 \|P_{N_1}\indl v_1\|_{L^2_{tx}}
\|P_{N_3}\indl v_3\|_{L^\infty_t L^2_x}
\|P_{N_2}Q_{\gtrsim L}v_2\|_{L^2_t L^\infty_x}
\|P_{N_4}v_4\|_{L^\infty_{tx}} \\
&  \les \omega_{N_1}\omega_{N_2} N_{(1)}^{-\alpha}N_2^\frac12 N_4^\frac12 \|\indl\|_{L^2_t} \|P_{N_1}v_1\|_{L^\infty_t L^2_x} \|P_{N_3} v_3 \|_{L^\infty_t L^2_x} \|P_{N_2}v_2 \|_{X^{0,1}} \|P_{N_4} v_4 \|_{L^\infty_t L^2_x},
\end{split}
\end{equation}
with a similar argument for ${\rm I}^{(2)}_{3,4}$. This completes the proof for \eqref{multline1} and {\it Case 2}. \\

\noindent
{\it \large Case 3: $N_1 \sim N_2 \gg N_4$}.
In this case, we split the symbol $A$ in \eqref{JT} as
\[
A(k_1,k_2,k_3,k_4) = B (k_1,k_2) + B (k_3,k_4),
\]
where 
$$
B (k_1,k_2) = \phi_N(k_2)^2|k_2|^{2\beta+2s}
- \phi_N(k_1)^2|k_1|^{2\beta+2s}.
$$
We can assume that $|k_1 - k_2| = |k_3 - k_4| \geq 1$ as otherwise $A \equiv 0$.
Let ${\rm J}_T^{(3,1)}$ be the restriction of ${\rm J}_T^{(3)}$ to the symbol $B(k_1,k_2)$.
\begin{align}
\label{3bound}
\begin{split}
{\rm J}_T^{(3,1)} & = \frac{1}{2} \sum_{\substack{N \sim N_1\sim N_2\gg N_4}} \omega_{N_1}^2 \bigg|  \int_0^T \sum_{k_1 - k_2 + k_3 - k_4 = 0} B(k_1,k_2) \\
& \hphantom{XXXXXXXX} \times \overline{\widehat{P_{N_4} u}}(k_4) \widehat{P_{N_3} u}(k_3) \overline{\widehat{ P_{N_2}u}}(k_2) \widehat{ P_{N_1} u}(k_1)  dt  \bigg|\\
& \les  \sum_{\substack{N \sim N_1\sim N_2\gg N_4}} \omega_{N_1}^2 \bigg|  \int_0^T \int_{\T} P_{N_4} \bar u   P_{N_3} u  \cdot K_{B} (  P_{N_2} \bar u, P_{N_1} u )  dx dt  \bigg|,
\end{split}
\end{align}
where $K_{B}$ is the multiplier with symbol $B$. We further note that $K_B$ gives rise to the commutator present in Lemma \ref{commutator} hence for $1/p + 1/q + 1/r = 1$, $1 \leq p,q,r \leq \infty$,
\begin{equation} \label{Bcommutator}
\bigg|\int_{\T}f g \cdot K_B(u,v) \, dx \bigg|\leq \|[D_x^{2\beta + 2s}P_N, fg](u) \cdot v \|_{L^1_x} \les N^{2\beta + 2s - 1}\|D_x(fg)\|_{L^p_{x}} \|u\|_{L^q_x}\|v\|_{L^r_x}.
\end{equation}
Hence we can push derivatives onto the product of smaller frequencies; this will be the core technique used in the argument of this case.
To proceed, we decompose the frequency of the product $P_{N_4} \bar u  P_{N_3} u $ and write
\begin{align}
{\rm J}_T^{(3,1)} \les \sum_{\substack{N \sim N_1\sim N_2\gg N_4 \\ M \les N_{(3)}}} \omega_{N_1}^2  \bigg| \int_0^T \int_\T P_M( P_{N_4}\bar u P_{N_3} u) \cdot K_{B} ( P_{N_2} \bar u ,P_{N_1} u) \,dx dt \bigg|.
\end{align}
We shall argue by showing the multilinear estimate
\begin{equation} \label{multline2}
\begin{split}
   & \omega_{N_{(1)}}^2 \bigg|\int_0^T \int_\T P_M(P_{N_{4}}\bar v_4 P_{N_3}v_3 ) \cdot K_B(P_{N_2} \bar v_2, P_{N_1} v_1 ) \, dx dt \bigg| \\ 
   & \qquad  \les T^\theta \omega_{N_{(1)}} \omega_{N_{(2)}} N_{(1)}^{2\beta +2s -\alpha}N_{(3)}^{\frac12}N_{(4)}^\frac12 \sum_{i=1}^4 \left(N_{(1)}^{0+}\|P_{N_i}v_i \|_{L^\infty_T L^2_x} + \|P_{N_i}v_i \|_{X^{0,1}_T} \right) \prod_{\substack{j = 1 \\ j \neq i}}^4 \|P_{N_j}v_j \|_{L^\infty_T L^2_x},
\end{split}
\end{equation}
where $\theta$ depends only on $\alpha$.
Assuming \eqref{multline2}, then
\begin{equation}
\begin{split}
{\rm J}_T^{(3,1)} &\les T^\theta\sum_{N_1 \sim N_2 \gg N_3, N_4}  \omega_{N_{(1)}}\omega_{N_{(2)}}N_{(1)}^{2\beta +2s -\alpha}N_{(3)}^{\frac12 +}N_{(4)}^\frac12 \\
& \qquad \times \sum_{i=1}^4 \left( N_{(1)}^{0+}N_{(i)}^{-s} \|P_{N_i}u \|_{L^\infty_T H^s} + N_{(i)}^{2\beta -s}\|P_{N_i}u \|_{X^{s-2\beta,1}_T} \right) \prod_{\substack{j = 1 \\ j \neq i}}^4 N_{(j)}^{-s}\|P_{N_j}u \|_{L^\infty_T H^s},
\end{split}
\end{equation}
where the sum in $M$ costs at most $N_{(3)}^{0+}$. The remaining dyadic
factors are
\[
N_{(1)}^{2\beta-\alpha+}N_{(3)}^{\frac12-s+}
N_{(4)}^{\frac12-s},
\qquad
N_{(1)}^{2\beta-\alpha}N_{(3)}^{\frac12-s+}
N_{(4)}^{\frac12-s}N_{(i)}^{2\beta}.
\]
Choosing the harmless plus exponents smaller than $s-\frac12$, the lower
frequency sums converge. The top frequency is controlled exactly as in
\eqref{boundConditions}, using $4\beta\le\alpha$.

We now prove \eqref{multline2}. By Lemma \ref{ResonanceRelation}, under these frequency projections on the LHS the resonance relation satisfies $|\Omega| \gtrsim N_1^{\alpha-1} M =: L$. Considering the LHS, where we again replace each $v_i$ by its extension to the real line from \eqref{extension}, we write 
\begin{equation}
\begin{split}
    &\omega_{N_{(1)}}^2 \bigg|\int_\R \int_\T P_M(P_{N_{4}}\bar v_4 P_{N_3}v_3 ) \cdot K_B(P_{N_2}\ind_T \bar v_2, P_{N_1} \ind_T v_1 ) \, dx dt \bigg| \\
    & \qquad \qquad= : {\rm I}^{(3,1)}_M(P_{N_1}\ind_T v_1, P_{N_2}\ind_T v_2, P_{N_3}v_3, P_{N_4}v_4 ).
\end{split}
\end{equation}
We proceed as in {\it Case 2} by decomposing the indicator functions $\ind_T$ as follows,
\begin{equation}
\begin{split}
{\rm I}^{(3,1)}_M & (P_{N_1}\ind_T v_1, P_{N_2}\ind_T v_2, P_{N_3}v_3, P_{N_4}v_4 ) \\
&\quad \leq {\rm I}^{(3,1)}_M (P_{N_1}\indh v_1, P_{N_2}\ind_T v_2, P_{N_3}v_3, P_{N_4}v_4 )  \\
&\quad +{\rm I}^{(3,1)}_M (P_{N_1}\indl v_1, P_{N_2}\indh v_2, P_{N_3}v_3, P_{N_4}v_4 )  \\
& \quad +{\rm I}^{(3,1)}_M (P_{N_1}\indl v_1, P_{N_2}\indl v_2, P_{N_3}v_3, P_{N_4}v_4 )\\
& =: {\rm I}_1^{(3,1)} + {\rm I}_2^{(3,1)} + {\rm I}_3^{(3,1)}.
\end{split}
\end{equation}
Choosing $R =   N_{(1)}^{\alpha-1-}M^{1+}$, using \eqref{indL1Bound} and \eqref{Bcommutator}
\begin{equation}
\begin{split}
{\rm I}_1^{(3,1)} &\les   \omega_{N_1}^2 N_{(1)}^{2\beta + 2s - 1} M  \|\ind_{T,R}^{\text{high}}\|_{L^1_t}\|P_{N_1}v_1\|_{L^\infty_t L^2_x}\|P_{N_2}v_2\|_{L^\infty_t L^2_x}  \|P_{M}(P_{N_3}v_3 P_{N_4}\bar v_4)\|_{L^\infty_{tx}} \\
& \les T^{\theta} \omega_{N_{(1)}}\omega_{N_{(2)}} N_{(1)}^{2\beta +2s -\alpha +} N_3^{\frac12} N_4^{\frac12}\|P_{N_1}v_1\|_{L^\infty_t L^2_x}\|P_{N_2}v_2\|_{L^\infty_t L^2_x}\|P_{N_3}v_3\|_{L^\infty_t L^2_x}\|P_{N_4}v_4\|_{L^\infty_t L^2_x},
\end{split}
\end{equation}
which is sufficient. A similar argument gives the same bound for ${\rm I}_2^{(3,1)}$. Turning to the term ${\rm I}_3^{(3,1)}$; we further decompose this term as before
\begin{equation}
\begin{split}
{\rm I}_3^{(3,1)} & \leq {\rm I}^{(3,1)}(P_{N_1}Q_{\gtrsim L}(\indl v_1),  P_{N_2}\indl v_2, P_{N_3}v_3, P_{N_4}v_4 )  \\
&\quad + {\rm I}^{(3,1)}(P_{N_1}Q_{\ll L}(\indl v_1),  P_{N_2}Q_{\gtrsim L}(\indl v_2), P_{N_3}v_3, P_{N_4}v_4 )\\
& \quad+ {\rm I}^{(3,1)}(P_{N_1}Q_{\ll L}(\indl v_1),  P_{N_2}Q_{\ll L}(\indl v_2), P_{N_3}Q_{\gtrsim L}v_3, P_{N_4}v_4 )\\
& \quad+ {\rm I}^{(3,1)}(P_{N_1}Q_{\ll L}(\indl v_1),  P_{N_2}Q_{\ll L}(\indl v_2), P_{N_3}Q_{\ll L} v_3, P_{N_4}Q_{\gtrsim L}v_4 )\\
& =: {\rm I}_{3,1}^{(3,1)} +{\rm I}_{3,2}^{(3,1)}+{\rm I}_{3,3}^{(3,1)}+{\rm I}_{3,4}^{(3,1)}.
\end{split}
\end{equation}
By \eqref{indLinftyBound}, \eqref{indL2Bound}, $L = N_1^{\alpha-1}M \gg R$, Lemma \ref{QLlow}, \eqref{Bcommutator} and Remark \ref{freqEnvolopeRatioBound}
\begin{equation}
\begin{split}
{\rm I}_{3,1}^{(3,1)} &\les \omega_{N_{(1)}}^2 N_{(1)}^{2s + 2\beta - 1} M \|P_{N_1}Q_{\gtrsim L}(\indl v_1)\|_{L^2_{tx}}\|P_{N_2}\indl v_2\|_{L^2_{tx}} \|P_{M}(P_{N_3}v_3 P_{N_4}\bar v_4)\|_{L^{\infty}_{tx}} \\
& \les \omega_{N_{(1)}}\omega_{N_{(2)}}N_{(1)}^{2s + 2\beta - \alpha} \|\indl\|_{L^2_t}\|P_{N_1}v_1\|_{X^{0,1}}\|P_{N_2}v_2\|_{L^\infty_t L^2_{x}}\|P_{M}(P_{N_3}v_3 P_{N_4}\bar v_4)\|_{L^{\infty}_{tx}} \\
& \les T^{\frac12} \omega_{N_{(1)}}\omega_{N_{(2)}} N_{(1)}^{2\beta +2s - \alpha} N_3^{\frac12} N_4^{\frac12} \|P_{N_1}v_1\|_{X^{0,1}}\|P_{N_2}v_2\|_{L^\infty_t L^2_x}\|P_{N_3}v_3\|_{L^\infty_t L^2_x} \|P_{N_4}v_4\|_{L^\infty_t L^2_x},
\end{split}
\end{equation}
which is sufficient. The same argument works for ${\rm I}_{3,2}^{(3,1)}$.

To simplify the next estimates observe by Cauchy-Schwarz
\begin{equation}
\begin{split} \label{JointModulationCS}
\|P_M(P_{N_3}Q_{\gtrsim L}v_3 P_{N_4}\bar v_4 )\|_{L^2_t L^\infty_x} & \les \sum_{K \gtrsim MN_{(1)}^{\alpha-1}}  \|P_M( P_{N_3}Q_{K}v_3 P_{N_4}\bar v_4)\|_{L^2_t L^\infty_x} \\
& \les \sum_{K \gtrsim MN_{(1)}^{\alpha-1}}  \|  P_{N_3}Q_{K}v_3 \|_{L^2_t L^\infty_x} \| P_{N_4}\bar v_4\|_{L^\infty_t L^\infty_x} \\
& \les N_{(1)}^{1-\alpha}M^{-1}N_3^\frac12 N_4^\frac12\|P_{N_3}v_3\|_{X^{0,1}}\|P_{N_4}v_4\|_{L^\infty_t L^2_x},
\end{split}
\end{equation}
Hence by \eqref{indLinftyBound}, \eqref{indL2Bound}, $L \gg R$, Lemma \ref{QLlow}, \eqref{Bcommutator}, Lemma \ref{QLBounded} and \eqref{JointModulationCS} we see
\begin{equation}
\begin{split}
{\rm I}_{3,3}^{(3,1)} 
& \les \omega_{N_{(1)}}^2 N_{(1)}^{2s+2\beta -1} M \|\indl\|_{L^2_t} \|P_{N_1}v_1\|_{L^\infty_t L^2_x} \|P_{N_2}v_2\|_{L^\infty_t L^2_x}  \|P_M(P_{N_3}Q_{\gtrsim L}v_3 P_{N_4}\bar v_4 )\|_{L^2_t L^\infty_x} \\
& \les T^{1/2} \omega_{N_{(1)}}\omega_{N_{(2)}} N_{(1)}^{2\beta +2s -\alpha}N_3^\frac12 N_4^\frac12 \|P_{N_1}v_1\|_{L^\infty_t L^2_x}\|P_{N_2}v_2\|_{L^\infty_t L^2_x}\|P_{N_3}v_3\|_{X^{0,1}}\|P_{N_4}v_4\|_{L^\infty_t L^2_x},
\end{split}
\end{equation}
where the same argument works for ${\rm I}^{(3,1)}_{3,4}$. This completes the proof of \eqref{multline2}.
The contribution of the $B(k_3,k_4)$ multiplier follows from the same
estimates after the permutation
$(k_1,k_2,k_3,k_4)\mapsto(k_3,k_4,k_1,k_2)$ and complex conjugation.  Under
this permutation, the commutator bound produces
\[\|D_xP_M(P_{N_1}v_1 P_{N_2} \bar v_2 )\|_{L^1_x} \]
in place of the corresponding $(N_3,N_4)$ term.  Thus the dyadic estimates
and summations already proved apply without change.
This completes the proof.
\end{proof}

\section{Estimate on differences}\label{sec-6}
This section is devoted to proving the estimate on differences for the solutions to \eqref{MMTEq} with regularity $s-1$.
In particular, we show the following result.

\begin{prop} \label{DiffEstimate}
Let $0 < T \leq 1, 1 < \alpha \leq 2, 0 < \beta \leq \alpha/4$ and 
\[s \ge \max( \gamma(\alpha,\beta), \beta + 1 -\frac\alpha2, 1/2+) .\] 
Let $u$ and $v$ be two solutions of \eqref{MMTEq} with initial data $u_0$ and $v_0 \in H^s(\T)$ respectively. Then there exists $\theta > 0 $ such that for $w = u - v$ 
\begin{equation}
\|w\|_{L^\infty_T H^{s-1}}^2 \les \|u_0 - v_0\|_{H^{s-1}}^2 + T^{\theta} \|w\|_{L^\infty_T H^{s-1}}\|w\|_{Z^{s-1}_T}(1+\|u\|_{Z^s_T}+\|v\|_{Z^s_T})^4,
\end{equation}
where $\theta$ and the implicit constants depend on $\alpha,\beta$, and $s$.
\end{prop}

\begin{proof}
Let $u$ and $v$ solve \eqref{MMTEq} with initial data $u_0$ and $v_0$,
respectively.
First, observe that the difference of two solutions $w = u - v$ satisfies
\begin{equation}
i\del_t w +(-\Delta)^{\alpha/2}w = \nu D_x^{2\beta}(|u|^2w + uv\bar{w} + |v|^2w).
\end{equation}
Following the same line of argument as we did at the beginning of the proof
of Proposition~\ref{Apriori}, for every dyadic $N\ge1$ we obtain the exact
dyadic energy identity, where
$\mathcal N(u,v,w):=|u|^2w+uv\bar w+|v|^2w$,
\begin{equation}
\begin{split}
N^{2(s-1)}\left(\|P_Nw(T)\|_{L^2}^2
-\|P_N(u_0-v_0)\|_{L^2}^2\right)
&=2\nu N^{2(s-1)}
\Im \bigg[ \int_0^T \int_\T
P_N^2D_x^{2\beta}\bar w\,\mathcal N(u,v,w)\,dxdt\bigg].
\end{split}
\end{equation}
By symmetry, it is sufficient to estimate terms of the form
\begin{equation} \label{DiffTerm1}
{\rm I} := N^{2(s-1)} \Im \bigg[ \int_0^T \int_\T P_N^2 D_x^{2\beta} \bar{w} \cdot|u|^2w \, dx dt \bigg]
\end{equation}
and
\begin{equation} \label{DiffTerm2}
{\rm I\!I} :=  N^{2(s-1)} \Im \bigg[ \int_0^T \int_\T P_N^2 D_x^{2\beta} \bar{w} \cdot uv \bar{w} \, dx dt \bigg].
\end{equation}
Before we estimate \eqref{DiffTerm1} and \eqref{DiffTerm2}, we first consider some special cases.

For $N \les 1$, we can easily bound \eqref{DiffTerm1} and \eqref{DiffTerm2} using Lemma \ref{SobolevMult} by
\begin{equation}
\begin{split}
T N^{2(s-1)}\|D_x^{2\beta}P_N \bar w w \|_{L^\infty_T H^{-\frac12}} \| |u|^2 \|_{L^\infty_T H^{\frac12}} &\les T N^{2(s+\beta-1)}\|P_{N}w\|_{L^\infty_T H^{\frac12+}}\|w\|_{L^\infty_T H^{-\frac12}}\|u\|_{L^\infty_T H^{\frac12+}}^2 \\
&\les T N^{2(s+\beta)-1+}\|w\|_{L^\infty_T H^{-\frac12}}^2\|u\|_{L^\infty_T H^{\frac12 + }}^2,
\end{split}
\end{equation}
which is sufficient provided $s > 1/2$. Hence we reduce to estimating for $N \gg 1$.

We now consider the resonant cases. Take \eqref{DiffTerm2} to begin, then
\[
\begin{split}
\int_\T P_N^2 D_x^{2\beta} \bar{w} \cdot u v \bar w \, dx
&=
\sum_{k_1 + k_2 - k_3 -k_4=0}
|k_4|^{2\beta}
\widehat u(k_1)
\widehat v(k_2)
\overline{\widehat w}(k_3)
\varphi_N^2(k_4)
\overline{\widehat w}(k_4)  \\
&=
\sum_{\substack{k_1 + k_2 - k_3-k_4=0\\
k_1\neq k_3,k_4}}
|k_4|^{2\beta}
\widehat u(k_1)
\widehat v(k_2)
\overline{\widehat w}(k_3)
\varphi_N^2(k_4)
\overline{\widehat w}(k_4) \\
&\quad+
\sum_{\substack{k_1 + k_2 - k_3-k_4=0\\
k_1=k_3\text{ or }k_1=k_4}}
|k_4|^{2\beta}
\widehat u(k_1)
\widehat v(k_2)
\overline{\widehat w}(k_3)
\varphi_N^2(k_4)
\overline{\widehat w}(k_4).
\end{split}
\]
The second term on the right-hand side contains the resonant contributions.
These terms can be handled directly. If $k_1=k_3$, then $k_2=k_4$, and
the corresponding contribution is bounded by
\begin{equation} \label{resonantContribution}
\begin{split}
N^{2(s-1+\beta)} &
\bigg|
\sum_{k_1}
\widehat u(k_1)\overline{\widehat w}(k_1)
\bigg|\cdot 
\bigg|
\sum_{|k_2|\sim N}
\widehat v(k_2)\overline{\widehat w}(k_2)
\bigg| \\
&\lesssim
N^{2\beta-1}
\|u\|_{H^s}\|v\|_{H^s}\|w\|_{H^{s-1}}^2,
\end{split}
\end{equation}
where we used $s>1/2$ to pair $H^s$ with $H^{s-1}$ and used the
frequency localization in the second factor.  More explicitly, after
summing over dyadic $N$, the remaining factor is
\[
\sum_N N^{2\beta-1}
\|P_Nv\|_{H^s}\|P_Nw\|_{H^{s-1}}
\lesssim \|v\|_{H^s}\|w\|_{H^{s-1}},
\]
by Cauchy--Schwarz and $2\beta-1\le0$.  If $k_1 = k_4$, then $k_2 = k_3$
and the same argument applies after interchanging the two pairs.

If instead we consider the resonant contributions of \eqref{DiffTerm1}, 
\begin{equation}
\begin{split}
\int_\T P_{N}^2 D_x^{2\beta}\bar w \cdot |u|^2 w \,dx &= \sum_{\substack{k_1- k_2 +k_3 - k_4 = 0 \\ k_1\neq k_2, k_4}}|k_4|^{2\beta} \hat{u}(k_1) \overline{\hat{u}}(k_2) \hat{w}(k_3) \phi_N^2(k_4) \overline{\hat{w}}(k_4) \\
&\quad + \sum_{\substack{k_1- k_2 +k_3 - k_4 = 0 \\ k_1=k_2 \text{ or } k_1 = k_4}}|k_4|^{2\beta} \hat{u}(k_1) \overline{\hat{u}}(k_2) \hat{w}(k_3) \phi_N^2(k_4) \overline{\hat{w}}(k_4).
\end{split}
\end{equation}
The resonant contributions are given by $k_1 = k_2$ or $k_1 = k_4$. If $k_1 =k_4$ then $k_2=  k_3$, we may argue exactly as in \eqref{resonantContribution}. If $k_1 = k_2$ then $k_3 = k_4$ and the summand is real valued and hence can be ignored due to taking the imaginary part in \eqref{DiffTerm1}.

In what follows, we therefore focus on the non-resonant cases, where $\Omega \neq 0$.

By projecting frequencies to $|k_i| \sim N_i$,
we first consider a special case where $N_4 \sim N_3 \sim N_2 \sim N_1$. 
In this case, we argue exactly as we did for the {\it a priori} estimate. Denoting by $u_i \in \{u, v\}$ for $i=1,2$, by H\"older, Corollary \ref{StrichartzSecondary} and \ref{Diffell4L4Bound}, we have
\begin{equation} \label{DiffCase1}
\begin{split}
|{\rm I}| +  |{\rm I\!I}| &\les \sum_{N \gg 1} N^{2s - 2 + 2\beta} \|\Tilde{P}_N w \|_{L^4_{Tx}} \|\Tilde{P}_N w \|_{L^4_{Tx}} \|\Tilde{P}_N u_1 \|_{L^4_{Tx}}\|\Tilde{P}_N u_2 \|_{L^4_{Tx}} \\
& \les \bigg(\sum_{N} \|D_x^{s -1- \kappa(\alpha,\beta)} \Tilde{P}_N w \|_{L^4_{Tx}}^4 \bigg)^{1/2}  \times \prod_{i=1}^2\bigg(\sum_{N} \|D_x^{ \beta +\kappa(\alpha,\beta)} \Tilde{P}_N u_i \|_{L^4_{Tx}}^4 \bigg)^{1/4} \\
& \les T^{1/2}\|w\|_{L^\infty_TH^{{s-1}}}^2 \|u_1\|_{L^{\infty}_T H^{\gamma(\alpha,\beta)}}\|u_2\|_{L^{\infty}_T H^{\gamma(\alpha,\beta)}}(1 + T(\|u_1\|_{L^\infty_T H^{1/2+}}+ \|u_2\|_{L^\infty_T H^{1/2+}})^2)^4
\end{split}
\end{equation}
which is sufficient provided $s \ge \max( \gamma(\alpha,\beta), 1/2+) $. As we did before in the proof of Proposition \ref{Apriori}, we denote by $N_{(1)} \geq N_{(2)} \geq N_{(3)} \geq N_{(4)}$ a size ordering on $N_1,N_2,N_3,N_4$.
By excluding the case treated above, we may assume that $N_{(4)} \ll N_{(1)}$.\\

\medskip

\subsection{First term estimate}~\\

We begin by considering the first term \eqref{DiffTerm1}.
In what follows, we distinguish two cases based on the size of the resonance relation $\Omega$;
\begin{itemize}
    \item {\it Case 1}: $|\Omega| \gtrsim N_{(1)}^\alpha$;
    \item {\it Case 2}: $1\le |\Omega| \ll N_{(1)}^{\alpha}$.
\end{itemize}
We begin considering {\it Case 1}.\\

\noindent
{\it \large Case 1: $|\Omega| \ges N_{(1)}^{\alpha}$}.\\

\noindent
In order to estimate \eqref{DiffTerm1}, it suffices to bound
\begin{align}
\label{Term1-0}
\begin{split}
{\rm I}
\lesssim
\sum_{N_1,N_2,N_3,N_4}
N_4^{2(s-1+\beta)}
\bigg|
\int_0^T\int_{\T}
P_{N_4}\bar w \,
P_{N_3}w\,
P_{N_2}\bar u\,
P_{N_1} u
\, dxdt
\bigg|.
\end{split}
\end{align}
We shall prove the following multilinear estimate:
\begin{align}
\label{Term1-1}
\begin{split}
\bigg| & 
\int_0^T\int_{\T}
 P_{N_1} v_1 P_{N_2}\bar v_2 P_{N_3} v_3 P_{N_4}\bar v_4
\, dxdt
\bigg| \\
& \lesssim
T^\theta N_{(1)}^{-\alpha}
\big(N_{(3)}N_{(4)}\big)^{\frac12}
\sum_{i=1}^4 \Big( N_{(1)}^{0+} \|P_{N_i}v_i\|_{L_T^\infty L^2}  + 
\|P_{N_i}v_i\|_{X^{0,1}} \Big)
\prod_{j\neq i}
\|P_{N_j}v_j\|_{L_T^\infty L^2},
\end{split}
\end{align}
where $\theta > 0$ depends only on $\alpha$ and each $v_i\in\{w, u\}$, provided $|\Omega| \ges N_{(1)}^\alpha$ and $N_{(4)} \ll N_{(1)}$.

Assuming \eqref{Term1-1}, we can prove Proposition \ref{DiffEstimate} for the term given in \eqref{DiffTerm1}. Indeed, applying \eqref{Term1-1} to \eqref{Term1-0} under conditions $N_{(1)} \sim N_{(2)}$ and $N_{(4)} \ll N_{(1)}$, we obtain the bound
\begin{align}
\begin{split}
{\rm I} &
\lesssim
\sum_{\substack{N_1,N_2,N_3,N_4\ge 1\\ N_{(1)}\sim N_{(2)}}}
 N_4^{s-1+2\beta}
N_3^{1-s}
N_1^{-s}
N_2^{-s}
N_{(1)}^{-\alpha + 2\beta}
\big(N_{(3)}N_{(4)}\big)^{\frac12} \\
&\quad \times\|P_{N_1} u\|_{Z^{s}_T} \|P_{N_2} u\|_{Z^{s}_T} 
\big( \|P_{N_3} w\|_{Z^{s-1}_T} \| P_{N_4} w\|_{L_T^\infty H^{s-1}} + \|P_{N_3} w\|_{L_T^\infty H^{s-1}} \| P_{N_4} w\|_{Z^{s-1}_T} \big)\\
&  \les \|u\|_{Z^s_T}^2 \| w\|_{L_T^\infty H^{s-1}} \| w\|_{Z^{s-1}_T} ,
\end{split}
\label{DyadicSumClaim}
\end{align}
provided $s > \frac12$ and $\alpha \ge 4\beta$. Note that we can take without loss of generality $\beta > 0$, and use $N_{(1)}^{0+} \leq N_{(1)}^{2\beta}$.

Indeed, recall
$N_{(1)}\sim N_{(2)}$.
We split the summation in \eqref{DyadicSumClaim} according to whether $N_4\sim N_{(1)}$ or $N_4\ll N_{(1)}$.
First, suppose that $N_4\sim N_{(1)}$. 
Since $N_{(1)}\sim N_{(2)}$, there is another frequency comparable to $N_{(1)}$. The worst case is when the compensating factor is $N_3^{1-s}$, namely $N_3\sim N_{(1)}$. Then
$N_4^{s-1+2\beta}N_3^{1-s}
\lesssim N_{(1)}^{2\beta}$.
Thus, the contribution for the highest frequency is
$N_{(1)}^{4\beta-\alpha }$. Overall we can reduce to factors
\[ N_{(1)}^{4\beta - \alpha}N_{(3)}^{\frac12 - s}N_{(4)}^{\frac12 - s},\]
and hence everything is summable provided $4\beta \leq \alpha$ and $s > \frac12$.

It remains to consider the case $N_4\ll N_{(1)}$. In this case, the power $N_4^{s-1+2\beta}$ is carried by a lower frequency. 
We need only consider the worst case where $N_3 \sim N_{(1)}$; if $N_1 \sim N_3 \sim N_{(1)}$ then $|\Omega| \ges N_{(1)}^\alpha$ under $N_4\ll N_{(1)}$ in view of Lemma \ref{ResonanceRelation}. Hence since the two largest frequencies must be comparable $N_{(1)} \sim N_2 \sim N_3 \gg N_1$, then
\[
N_4^{s-1+2\beta}
N_3^{1-s}
N_1^{-s}
N_2^{-s}
N_{(1)}^{-\alpha + 2\beta}
\big(N_{(3)}N_{(4)}\big)^{\frac12} \les N_{(1)}^{1-2s -\alpha+2\beta} N_4^{s-\frac12+2\beta} N_{(4)}^{\frac12 -s}
\les N_{(1)}^{\frac12 -s -\alpha + 4\beta},
\]
which is summable under the slightly stronger condition $s > \frac12$ and $\alpha \ge 4\beta$.
This completes the proof of \eqref{DyadicSumClaim}.

\medskip 

To complete this case for the term \eqref{DiffTerm1}, it remains to prove \eqref{Term1-1}. 
In what follows, we focus on this estimate. Without loss of generality, we may assume that
$N_1\ge N_2\ge N_3\ge N_4$
in \eqref{Term1-1}.
Define
\begin{align}\label{I1}
{\rm I}^{(1)}(P_{N_1}v_1, P_{N_2}v_2, P_{N_3}v_3, P_{N_4}v_4) := \bigg|
\int_0^T\int_{\T}
P_{N_1} v_1 P_{N_2}\bar v_2 P_{N_3} v_3 P_{N_4} \bar v_4 \, dxdt \bigg|.
\end{align}
where implicitly we have assumed that $|\Omega| \gtrsim N_{(1)}^\alpha$
and $N_{(4)} \ll N_{(1)}$.
After taking extensions given by \eqref{extension}, still denoted by $v_i$,
it is sufficient to consider bounding a term of the form
\begin{equation}\label{J-21}
\begin{split}
 &  {\rm I}^{(1)}(P_{N_1}\ind_T v_1,P_{N_2}\ind_T v_2, P_{N_3}v_3, P_{N_4}v_4) = \bigg| \int_\R \int_\T  P_{N_1}\ind_T v_1 P_{N_2}\ind_T\bar v_2  P_{N_3}  v_3  P_{N_4} \bar v_4 \, dxdt \bigg| .
\end{split}
\end{equation}
Decomposing indicator functions
\begin{equation}
\begin{split}
{\rm I}^{(1)} & (P_{N_1}\ind_T v_1,P_{N_2}\ind_T v_2, P_{N_3} v_3, P_{N_4}v_4)\\
& \leq {\rm I}^{(1)} (P_{N_1}\indh  v_1,P_{N_2}\ind_T v_2, P_{N_3} v_3, P_{N_4} v_4) \\
&\quad + {\rm I}^{(1)} (P_{N_1}\indl v_1,P_{N_2} \indh v_2, P_{N_3} v_3, P_{N_4} v_4) \\
& \quad +  {\rm I}^{(1)} (P_{N_1}\indl v_1,P_{N_2}\indl v_2, P_{N_3} v_3, P_{N_4} v_4) \\
& =: {\rm I}^{(1)}_1 + {\rm I}_2^{(1)} + {\rm I}_3^{(1)}.
\end{split}
\end{equation}
With choice $R = N_{(1)}^{\alpha -}$, by Lemma \ref{time-multi} and H\"older,
we have
\begin{equation}
\begin{split}
{\rm I}^{(1)}_1 + {\rm I}^{(1)}_2 &\les (N_3N_4)^{\frac12} \|\indh\|_{L^1_t}\|P_{N_4} v_4\|_{L^\infty_t L^2_x}\|P_{N_3} v_3\|_{L^\infty_t L^2_x}\|P_{N_2} v_2\|_{L^\infty_t L^2_x}\|P_{N_1} v_1\|_{L^\infty_t L^2_x} \\
& \les T^\theta N_1^{-\alpha+} (N_3N_4)^{\frac12} \prod_{i=1}^4 \|P_{N_i} v_i\|_{L^\infty_t L^2_x},
\end{split}
\end{equation}
which is sufficient for \eqref{Term1-1}.

For the term $ {\rm I}_{3}^{(1)} $, we introduce the following modulation decomposition
\begin{equation}
\begin{split}
{\rm I}_{3}^{(1)}  & \leq {\rm I}_{3}^{(1)} (P_{N_1}Q_{\gtrsim L}\indl  v_1,P_{N_2}\indl v_2, P_{N_3}v_3, P_{N_4} v_4) \\
&\quad + {\rm I}_{3}^{(1)} (P_{N_1}Q_{\ll L}\indl v_1,P_{N_2}Q_{\gtrsim L}\indl v_2, P_{N_3} v_3, P_{N_4} v_4) \\
&\quad + {\rm I}_{3}^{(1)} (P_{N_1}Q_{\ll L}\indl v_1,P_{N_2}Q_{\ll L}\indl v_2, P_{N_3}Q_{\gtrsim L}v_3, P_{N_4} v_4) \\
& \quad+ {\rm I}_{3}^{(1)} (P_{N_1}Q_{\ll L} \indl v_1,P_{N_2}Q_{\ll L}\indl v_2, P_{N_3}Q_{\ll L}  v_3, P_{N_4}Q_{\gtrsim L} v_4) \\ 
& =: {\rm I}_{3,1}^{(1)} + {\rm I}_{3,2}^{(1)}+ {\rm I}_{3,3}^{(1)}+ {\rm I}_{3,4}^{(1)}.
\end{split}
\end{equation}
By H\"older, Lemma \ref{time-multi}, $L := N_{(1)}^{\alpha} \gg R$ and Lemma \ref{QLlow}, we have
\begin{equation}
\begin{split}
{\rm I}_{3,1}^{(1)} & \les \|\indl\|_{L^2_t} \|P_{N_1} Q_{\gtrsim L} v_1\|_{L^2_t L^2_x} \|P_{N_2}  v_2\|_{L^\infty_t L^2_x}\|P_{N_3}v_3\|_{L^\infty_{tx}}\|P_{N_4} v_4\|_{L^\infty_{tx}} \\
& \quad \les T^{1/2} N_{(1)}^{-\alpha} (N_3N_4)^{\frac12} \|P_{N_1}v_1\|_{X^{0,1} }\|P_{N_2} v_2\|_{L^\infty_t  L^2_x} \|P_{N_3}v_3 \|_{L^\infty_t  L^2_x}\|P_{N_4}v_4\|_{L^\infty_t  L^2_x},
\end{split}
\end{equation}
which implies \eqref{Term1-1}.
Similarly, we have
\begin{equation}
\begin{split}
 {\rm I}_{3,2}^{(1)} & \les \|\indl\|_{L^2_t} \|P_{N_1} v_1\|_{L^\infty_t L^2_x} \|P_{N_2} Q_{\gtrsim L}  v_2\|_{L^2_t L^2_x}\|P_{N_3}v_3\|_{L^\infty_{tx}}\|P_{N_4} v_4\|_{L^\infty_{tx}} \\
& \quad \les T^{1/2} N_{(1)}^{-\alpha} (N_3N_4)^{\frac12} \|P_{N_1}v_1\|_{L^\infty_t  L^2_x }\|P_{N_2} v_2\|_{X^{0,1}} \|P_{N_3}v_3 \|_{L^\infty_t  L^2_x}\|P_{N_4}v_4\|_{L^\infty_t  L^2_x}.
\end{split}
\end{equation}
For $ {\rm I}_{3,3}^{(1)}$, using Lemma \ref{QLlow} gets
\begin{equation}
\begin{split}
 {\rm I}_{3,3}^{(1)} & \les \|\indl\|_{L^2_t} \|P_{N_4}v_4\|_{L^\infty_{tx}} \|P_{N_3}  Q_{\gtrsim L} (\indl  v_3) \|_{L^2_{t} L^\infty_x}\|P_{N_2}  v_2\|_{L^\infty_t L^2_x} \|P_{N_1} v_1\|_{L^\infty_t L^2_x} \\
 & \les T^{1/2} (N_3N_4)^{\frac12} \|P_{N_4}v_4\|_{L^\infty_t L^2_x} \|P_{N_3}  Q_{\gtrsim L}  v_3\|_{L^2_{tx}} \|P_{N_2}  v_2\|_{L^\infty_t L^2_x} \|P_{N_1} v_1\|_{L^\infty_t L^2_x} \\
& \quad \les T^{1/2} N_{(1)}^{-\alpha} (N_3N_4)^{\frac12} \|P_{N_4}v_4\|_{L^\infty_t L^2_x }\|P_{N_3} v_3\|_{X^{0,1}} \|P_{N_2}v_2 \|_{L^\infty_t  L^2_x}\|P_{N_1}v_1\|_{L^\infty_t  L^2_x}.
\end{split}
\end{equation}
Similarly, we have
\begin{equation}
\begin{split}
 {\rm I}_{3,4}^{(1)} & \les \|\indl\|_{L^2_t} \|P_{N_4} Q_{\gtrsim L} (\indl  v_4)\|_{L^2_{t} L^\infty_x} \|P_{N_3} v_3\|_{L^\infty_{tx}} \|P_{N_2}  v_2\|_{L^\infty_t L^2_x}\|P_{N_1} v_1\|_{L^\infty_t L^2_x} \\
 & \les T^{\frac12} (N_3N_4)^{\frac12} \|P_{N_4} Q_{\gtrsim L}  v_4\|_{L^2_{tx}} \|P_{N_3}v_3\|_{L^\infty_t L^2_x}\|P_{N_2}v_2\|_{L^\infty_t L^2_x}\|P_{N_1}v_1\|_{L^\infty_t L^2_x} \\
& \quad \les T^{1/2} N_{(1)}^{-\alpha} (N_3N_4)^{\frac12} \|P_{N_4}v_4\|_{X^{0,1}}\|P_{N_3} v_3\|_{L^\infty_t L^2_x } \|P_{N_2}v_2 \|_{L^\infty_t  L^2_x}\|P_{N_1}v_1\|_{L^\infty_t  L^2_x}.
\end{split}
\end{equation}
All combined this proves \eqref{Term1-1} for the term ${\rm I}_{3}^{(1)}$, and hence completes the proof of {\it Case 1}. \\

\noindent
{\it \large Case 2: $1 \leq |\Omega| \ll N_{(1)}^{\alpha}$}.\\

We first note that
\begin{equation}
\begin{split}
2 N^{2(s-1)}\Im & \bigg[ \int_0^T \int_\T P_N^2 D_x^{2\beta} \bar{w} \cdot|u|^2w \, dx dt \bigg] \\
& = 2N^{2(s-1)}\Im \bigg[ \int_0^T \sum_{k_1 - k_2 + k_3 - k_4 = 0}|k_4|^{2\beta}\phi_N^2(k_4)\bar{\hat{w}}(k_4) \hat{w}(k_3) \bar{\hat{u}}(k_2) \hat{u}(k_1) \bigg].
\end{split}
\end{equation}
After frequency decomposition, if $|k_1 - k_2| \ll N_{(1)}$ and
$|k_2 - k_3| \ll N_{(1)}$, then
$|k_1| \sim |k_2| \sim |k_3| \sim |k_4|$, which reduces to the special
case treated above. If instead
$|k_1 - k_2|, |k_2 - k_3| \gtrsim N_{(1)}$, Lemma
\ref{ResonanceRelation} returns us to {\it Case 1}. We therefore treat the
two remaining subcases separately. \\

\noindent
{\it \large Case 2.1: $|k_1 - k_2| \ll N_{(1)}$ and $|k_2 - k_3| \gtrsim N_{(1)}$.}\\

We further observe that
\[
\begin{split}
2 N^{2(s-1)}\Im & \bigg[ \int_0^T \int_\T P_N^2 D_x^{2\beta} \bar{w} \cdot|u|^2w \, dx dt \bigg] \\
& =  N^{2(s-1)}\Im \bigg[ \int_0^T \int_\T P_N^2 D_x^{2\beta} \bar{w} \cdot|u|^2w - \bar{w} \cdot|u|^2 P_N^2 D_x^{2\beta} w \, dx dt \bigg] \\
& =  N^{2(s-1)}\Im \bigg[ \int_0^T \sum_{k_1 - k_2 + k_3 - k_4 = 0} B(k_4,k_3) \bar{\widehat{w}}(k_4) \widehat{w}(k_3) \bar{\widehat{ u}}(k_2) \widehat{ u}(k_1)
dt \bigg] \\
& =  N^{2(s-1)}\Im \bigg[ \int_0^T \int_{\T} K_B ( {\bar w},  w) \bar u u  \,d x
dt \bigg],
\end{split}
\]
where $K_B$ is the multiplier with symbol $B$ defined similarly as before in the proof of Proposition \ref{Apriori}, specifically
\[
B(k_4,k_3) = \phi_N^2 (k_4) |k_4|^{2\beta} - \phi_N^2 (k_3) |k_3|^{2\beta}.
\]
As before, by Lemma \ref{commutator}, the multiplier $K_B$ satisfies, for $1/p + 1/q + 1/r = 1$ and $1 \leq p,q,r \leq \infty$,
\begin{equation} \label{B2Commutator}
\bigg|\int_\T fg \cdot K_B(u,v) \,dx\bigg| \leq \|[P_ND_x^{2\beta}, fg](u) \cdot v\|_{L^1_x} \les N^{2\beta - 1}\|D_x(fg)\|_{L^p_x}\|u\|_{L^q_x}\|v\|_{L^r_x}.
\end{equation}
After dyadic decomposition, we can write \eqref{DiffTerm1} as
\begin{equation}\label{J-1T}
{\rm I} \les \sum_{N \gg 1}\sum_{N_1,N_2,N_3,N_4} N^{2(s-1)}\Im \bigg[\int_0^T \int_{\T} K_B (P_{N_4} \bar w, P_{N_3}{w})\cdot P_{N_2}\bar u P_{N_1} u \, dx dt\bigg] . 
\end{equation}
We argue by showing the following multilinear estimate.
\begin{align}
\label{Term2-1}
\begin{split}
\bigg| &
\int_0^T\int_{\T}
K_B (P_{N_4} \bar v_4, P_{N_3} v_3) P_{N_2} \bar v_2 P_{N_1}v_1
\, dxdt
\bigg| \\
& \lesssim
T^\theta N^{2\beta-1}N_{(1)}^{1-\alpha}
\big(N_{(3)}N_{(4)}\big)^{\frac12}
\sum_{i=1}^4 \left(N_{(1)}^{0+}\|P_{N_i}v_i\|_{L^\infty_T L^2_x} + 
\|P_{N_i}v_i\|_{X^{0,1}_T} \right)
\prod_{j\neq i}
\|P_{N_j}v_j\|_{L_T^\infty L^2_x}.
\end{split}
\end{align}
Again without loss of generality we take $\beta > 0$. If we assume \eqref{Term2-1}, then 
\begin{equation}
\begin{split}
{\rm I} &\les T^\theta \sum_{\substack{N_1,N_2,N_3,N_4 \geq 1 \\ N_{(1)}\sim N_{(2)} \\ N_4 \sim N \text{ or } N_3 \sim N}} N^{2(s+\beta-1)-1}N_{(1)}^{2\beta+1-\alpha-2s}(N_{(3)}N_{(4)})^{\frac12 - s} N_3N_4\\
& \quad \times \|P_{N_1}u\|_{Z^{s}_T}\|P_{N_2}u\|_{Z^s_T}\left(\|P_{N_3}w\|_{Z^{s-1}_T}\|P_{N_4}w\|_{L^\infty_T H^{s-1}} + \|P_{N_3}w\|_{L^\infty_T H^{s-1}}\|P_{N_4}w\|_{Z^{s-1}_T} \right) \\
& \les T^\theta \|u\|_{Z^s_T}^2 \|w\|_{L^\infty_T H^{s-1}} \|w\|_{Z^{s-1}_T},
\end{split}
\end{equation}
provided $s > \frac12$, $\alpha \geq 4\beta $ and $s \geq \beta + 1- \frac\alpha2$.

Indeed splitting into cases $N_{(1)} \sim N_3 \sim N_4 \gtrsim N_1, N_2$ or $N_{(1)} \sim N_1 \sim N_2 \gtrsim N_3,N_4$, in the former after collecting factors
\[  N_{(1)}^{4\beta -\alpha}N_1^{\frac{1}{2}-s}N_2^{\frac12 -s },\]
we obtain summability for $\alpha \geq 4\beta $ and $s > \frac12$. In the latter, without loss of generality $N \sim N_4$,
\begin{equation}
N_{(1)}^{2\beta + 1-\alpha - 2s} N_3^{\frac32 - s}N_4^{2(s+\beta -1) +\frac12 - s} \les N_{(1)}^{2\beta + 2 - \alpha - 2s} N_3^{\frac12 -s} N_4^{2(s+\beta - 1) + \frac12 -s} \les N_3^{\frac12 -s}N_4^{\frac12 -s}N_4^{4\beta-\alpha},
\end{equation}
which is summable provided $s \geq \beta + 1-\frac\alpha2$,
$\alpha \geq 4\beta$, and $s> \frac12$.

To complete this subcase, we show \eqref{Term2-1}. First we denote the LHS of \eqref{Term2-1} by
\begin{align}
{\rm I}^{(2.1)} (v_1,v_2,v_3,v_4) := \bigg| \int_0^T \int_{\T} K_B ( \bar v_4, v_3) \bar v_2 v_1  \,dx dt \bigg|,
\end{align}
and further decompose into 
\begin{align}
\begin{split}
{\rm I}^{(2.1)} & (P_{N_1} v_1, P_{N_2} v_2, P_{N_3} v_3,  P_{N_4} v_4) \\
& \les \sum_{M \le N_{(3)}}  \bigg| \int_0^T \int_{\T}  K_B (P_{N_4} \bar v_4, P_{N_3}  v_3) P_M (P_{N_2} \bar v_2 P_{N_1} v_1) dxdt \bigg|. 
\end{split}
\end{align}
By the position of the frequency projection $P_M$, that is since $N_{(1)} \gg|k_2 - k_1| \sim M$, we must have $M \les N_{(3)}$
and thus the summation in $M$ contributes at most a factor of $\log N_{(3)}$, which is harmless, and can be omitted.

Let us consider the integral, introducing real line extensions and indicator functions for the time interval as we have done before,
\begin{align}
&{\rm I}^{(2.1,M)} (P_{N_1} \ind_T v_1, P_{N_2}\ind_T v_2, P_{N_3} v_3,  P_{N_4} v_4) \\
& \quad := \bigg| \int_\R \int_{\T} K_B (P_{N_4}\bar  v_4, P_{N_3} v_3) P_M(P_{N_2} \ind_T \bar v_2 P_{N_1} \ind_T v_1 ) \,dxdt \bigg|. 
\end{align}
Under these frequency projections and the current restrictions imposed
we have $|\Omega| \ges N_{(1)}^{\alpha - 1}M$ for each $M$. Let $v_{(i)}$ be term corresponding to frequency projection $P_{N_{(i)}}$. 
As in previous cases we similarly decompose:
\begin{equation}
\begin{split}
{\rm I}^{(2.1,M)} & (P_{N_1} \ind_T v_1,P_{N_2} \ind_T v_2, P_{N_3} v_3, P_{N_4} v_4)\\
& \leq {\rm I}^{(2.1,M)} (P_{N_1}\indh v_1,P_{N_2}\ind_T v_2, P_{N_3} v_3, P_{N_4} v_4) \\
&\quad + {\rm I}^{(2.1,M)} (P_{N_1}\indl v_1,P_{N_2}\indh v_2, P_{N_3} v_3, P_{N_4} v_4) \\
& \quad +  {\rm I}^{(2.1,M)} (P_{N_1}\indl v_1,P_{N_2}\indl v_2, P_{N_3} v_3, P_{N_4} v_4) \\
& =: {\rm I}^{(2.1,M)}_1 + {\rm I}_2^{(2.1,M)} + {\rm I}_3^{(2.1,M)}.
\end{split}
\end{equation}
Then by choosing $R = N_{(1)}^{\alpha -1-} M^{1+} \ll N_{(1)}^{\alpha -1} M $, \eqref{B2Commutator},\eqref{indL1Bound}, Bernstein and H\"older inequalities,
we have
\begin{align}
\begin{split}
{\rm I}^{(2.1,M)}_1 & (P_{N_1}\indh v_1,P_{N_2}\ind_T v_2, P_{N_3}v_3, P_{N_4} v_4) \\
& \les N^{2\beta-1} M \|\indh\|_{L^1_t} \|P_{N_{(1)}} v_{(1)}\|_{L^\infty_t L^2_x} \|P_{N_2} v_{(2)}\|_{L^\infty_t L^2_x} \| P_{N_{(3)}} v_{(3)}\|_{L^\infty_{tx}}\| P_{N_{(4)}}  v_{(4)} \|_{L^\infty_{tx}} \\
& \les T^\theta N^{2\beta-1} N_{(1)}^{1-\alpha +} \|P_{N_{(1)}} v_{(1)}\|_{L^\infty_t L^2_x} \|P_{N_2} v_{(2)}\|_{L^\infty_t L^2_x} \| P_{N_{(3)}} v_{(3)}\|_{L^\infty_{tx}} \|P_{N_{(4)}}   v_{(4)} \|_{L^\infty_{tx}}
\end{split}
\end{align}
which is sufficient for \eqref{Term2-1}.
With a similar argument for ${\rm I}^{(2.1,M)}_2$, together they imply
\begin{equation}
\begin{split}
& {\rm I}^{(2.1,M)}_1 + {\rm I}^{(2.1,M)}_2 \\
&\les T^\theta N^{2\beta-1}N_{(1)}^{1-\alpha +} \|P_{N_{(4)}} v_{(4)}\|_{L^\infty_{tx} }\|P_{N_{(3)}} v_{(3)}\|_{L^\infty_{tx}}\|P_{N_{(2)}} v_{(2)}\|_{L^\infty_t L^2_x}\|P_{N_{(1)}} v_{(1)}\|_{L^\infty_t L^2_x} \\
&\les T^\theta N^{2\beta -1}  N_{(1)}^{1-\alpha+} (N_{(3)} N_{(4)})^{\frac12} \prod_{i=1}^4 \|P_{N_{i}} v_{i}\|_{L^\infty_t L^2_x} \\
\end{split}
\end{equation}
which is again sufficient for \eqref{Term2-1}.

For the term $ {\rm I}_{3}^{(2.1,M)} $, we introduce the following modulation decomposition
\begin{equation}
\begin{split}
{\rm I}_{3}^{(2.1,M)}  & \leq {\rm I}_{3}^{(2.1,M)} (P_{N_1}Q_{\gtrsim L}\indl v_1,P_{N_2}\indl v_2, P_{N_3} v_3, P_{N_4} v_4) \\
&\quad + {\rm I}_{3}^{(2.1,M)} (P_{N_1}Q_{\ll L}\indl v_1,P_{N_2}Q_{\gtrsim L}\indl v_2, P_{N_3} v_3, P_{N_4} v_4) \\
&\quad + {\rm I}_{3}^{(2.1,M)} (P_{N_1}Q_{\ll L}\indl v_1,P_{N_2}Q_{\ll L}\indl v_2, P_{N_3}Q_{\gtrsim L} v_3, P_{N_4} v_4) \\
& \quad+ {\rm I}_{3}^{(2.1,M)} (P_{N_1}Q_{\ll L}\indl  v_1,P_{N_2}Q_{\ll L}\indl v_2, P_{N_3}Q_{\ll L} v_3, P_{N_4}Q_{\gtrsim L} v_4) \\ 
& =: {\rm I}_{3,1}^{(2.1,M)} + {\rm I}_{3,2}^{(2.1,M)}+ {\rm I}_{3,3}^{(2.1,M)}+ {\rm I}_{3,4}^{(2.1,M)}.
\end{split}
\end{equation}
We consider the case $N_1 \sim  N_2 \gtrsim N_3, N_4$.  The case
$N_3 \sim N_4 \gtrsim N_1,N_2$ follows after exchanging the pairs
$(v_1,v_2)$ and $(v_3,v_4)$.  In either ordering, Bernstein's inequality is
applied only to the two lower-frequency factors.
By H\"older, \eqref{indL2Bound}, $L := N_{(1)}^{\alpha-1} M \gg R$ and Lemma \ref{QLlow}
\begin{equation}
\begin{split}
 {\rm I}_{3,1}^{(2.1,M)} & \les N^{2\beta -1} M \|\indl\|_{L^2_t} \|P_{N_1}  Q_{\gtrsim L} v_1\|_{ L^2_{tx}} \|P_{N_2}  v_2\|_{L^\infty_t L^2_x}\|P_{N_3}v_3\|_{L^\infty_{tx}}\|P_{N_4} v_4\|_{L^\infty_{tx}} \\
& \les T^{1/2} N^{2\beta-1}N_{(1)}^{1-\alpha} (N_3N_4)^{\frac12} \|P_{N_4}v_4\|_{L^\infty_t L^2_x }\|P_{N_3} v_3\|_{L^\infty_t  L^2_x} \|P_{N_2}v_2 \|_{L^\infty_t  L^2_x}\|P_{N_1}v_1\|_{X^{0,1}},
\end{split}
\end{equation}
which implies \eqref{Term2-1}.
Similarly, we have
\begin{equation}
\begin{split}
 {\rm I}_{3,3}^{(2.1,M)} & \les N^{2\beta -1} M \|\indl\|_{L^2_t} \|P_{N_1}  v_1\|_{L^\infty_t L^2_x } \|P_{N_2}  v_2\|_{L^\infty_t L^2_{x}} \|P_{N_3}Q_{\gtrsim L}v_3\|_{L^2_t L^\infty_{x}}\|P_{N_4}  v_4\|_{L^\infty_{tx}} \\
& \les T^{1/2} N^{2\beta-1}N_{(1)}^{1-\alpha} (N_3N_4)^{\frac12} \|P_{N_4}v_4\|_{L^\infty_t L^2_x }\|P_{N_3} v_3\|_{X^{0,1}} \|P_{N_2}v_2 \|_{L^\infty_t L^2_x}\|P_{N_1}v_1\|_{L^\infty_t  L^2_x},
\end{split}
\end{equation}
which is again sufficient for our purpose.  The bounds for
${\rm I}_{3,2}^{(2.1,M)}$ and ${\rm I}_{3,4}^{(2.1,M)}$ are obtained from
the displayed estimates by placing the $X^{0,1}$ norm on $v_2$ and $v_4$,
respectively; all frequency factors are unchanged.\\

\noindent
{\it \large Case 2.2: $|k_1 - k_2| \gtrsim N_{(1)}$ and $|k_2 - k_3| \ll N_{(1)}$.}\\

Under the given frequency restrictions, we shall prove the following multilinear estimate
\begin{equation}
\begin{split} \label{Term2-2}
&\bigg|\int_0^T \int_\T P_{N_1}v_1  P_{N_2}\bar v_2 P_{N_3}v_3 P_{N_4}\bar v_4\,dxdt \bigg| \\
& \quad \les T^\theta N_{(1)}^{1-\alpha} N_{(3)}^{-\frac12}N_{(4)}^{\frac12 +} \sum_{i = 1}^4 \big( N_{(1)}^{0+}\|P_{N_i}v_i\|_{L^\infty_T L^2_x} + \|P_{N_i}v_i\|_{X^{0,1}_T} \big) \prod_{j \neq i} \|P_{N_j} v_j \|_{L^\infty_T L^2_x}.
\end{split}
\end{equation}
If we assume \eqref{Term2-2} then
\begin{equation}
\begin{split}
{\rm I} &\les T^\theta \sum_{\substack{N_1,N_2,N_3,N_4 \geq 1 \\ N_{(1)} \sim N_{(2)} \\ N_4 \sim N \text{ or } N_3 \sim N}} N^{2(s+\beta - 1)} N_{(1)}^{2\beta +1-\alpha -2s} N_{(3)}^{-\frac12 -s}N_{(4)}^{\frac12 -s +} N_3N_4\\
& \quad \times \|P_{N_1}u\|_{Z^{s}_T} \|P_{N_2}u\|_{Z^s_T}\big(\|P_{N_3}w\|_{Z^{s-1}_T}\|P_{N_4}w\|_{L^\infty_T H^{s-1}} + \|P_{N_3}w\|_{L^\infty_T H^{s-1}}\|P_{N_4}w\|_{Z^{s-1}_T} \big) \\
&\les T^\theta \|u\|_{Z^s_T}^2 \|w\|_{L^\infty_T H^{s-1}} \|w\|_{Z^{s-1}_T},
\end{split}
\end{equation}
provided $s  > \frac12$, $\alpha \geq 4\beta$ and $s \geq \beta + 1 -\frac\alpha2$. Indeed splitting into cases $N_{(1)} \sim N_2 \sim N_3 \gtrsim N_1,N_4$ or $N_{(1)} \sim N_1 \sim N_4 \gtrsim N_2, N_3$, by collecting factors in the former case, if $N \sim N_4$

\begin{equation}
N_{(3)}^{\frac12 -s}N_{(4)}^{\frac12 -s +} N_4^{2(s+\beta-1)}N_3^{2\beta +2-\alpha-2s} 
\les N_4^{4\beta -\alpha} N_{(3)}^{\frac12 -s }N_{(4)}^{\frac12 -s +}.
\end{equation}
With similar for $N \sim N_3$, in both cases $s > \frac12$, $\alpha \geq 4\beta$ and $s \geq \beta + 1- \frac\alpha2$ is sufficient for summability. The latter case is analogous.

It remains to prove \eqref{Term2-2}. By denoting the LHS by
\begin{equation}
{\rm I}^{(2.2)}(P_{N_1}v_1, P_{N_2}v_2, P_{N_3}v_3, P_{N_4}v_4) := \bigg|\int_0^T \int_\T P_{N_1}v_1 P_{N_2}\bar v_2 P_{N_3} v_3  P_{N_4} \bar v_4 \, dxdt\bigg|
\end{equation}
and further decompose
\begin{equation}
\begin{split}
&{\rm I}^{(2.2)}(P_{N_1}v_1, P_{N_2}v_2, P_{N_3}v_3, P_{N_4}v_4) \\
& \qquad \les \sum_{M \les N_{(3)}} \bigg|\int_0^T \int_\T P_{N_1}v_1 P_M(P_{N_2}\bar v_2 P_{N_3} v_3 ) P_{N_4} \bar v_4 \, dxdt\bigg|.
\end{split}
\end{equation}
Again the summation in $M$ contributes at most a factor of $\log N_{(3)}$ and can be omitted. By replacing by extensions \eqref{extension}, we consider the integral
\begin{equation}
\begin{split}
&{\rm I}^{(2.2,M)}(P_{N_1}\ind_T v_1, P_{N_2}\ind_T v_2, P_{N_3}v_3, P_{N_4}v_4) \\
& \qquad := \bigg|\int_\R \int_\T P_{N_1} \ind_T v_1 P_M(P_{N_2} \ind_T \bar v_2 P_{N_3} v_3 ) P_{N_4} \bar v_4 \, dxdt\bigg|.
\end{split}
\end{equation}
As before, let $v_{(i)}$ denote the term corresponding to frequency projection $P_{N_{(i)}}$ and decompose,
\begin{equation}
\begin{split}
{\rm I}^{(2.2,M)} & (P_{N_1} \ind_T v_1,P_{N_2} \ind_T v_2, P_{N_3} v_3, P_{N_4} v_4)\\
& \leq {\rm I}^{(2.2,M)} (P_{N_1}\indh v_1,P_{N_2}\ind_T v_2, P_{N_3} v_3, P_{N_4} v_4) \\
&\quad + {\rm I}^{(2.2,M)} (P_{N_1}\indl v_1,P_{N_2}\indh v_2, P_{N_3} v_3, P_{N_4} v_4) \\
& \quad +  {\rm I}^{(2.2,M)} (P_{N_1}\indl v_1,P_{N_2}\indl v_2, P_{N_3} v_3, P_{N_4} v_4) \\
& =: {\rm I}^{(2.2,M)}_1 + {\rm I}_2^{(2.2,M)} + {\rm I}_3^{(2.2,M)}.
\end{split}
\end{equation}
Estimating the first term, choose
$R=N_{(1)}^{\alpha-1-}M^{1+}$, so that
$R\ll L:=N_{(1)}^{\alpha-1}M$.  Using \eqref{indL1Bound}, Lemma
\ref{SobolevMult}, Bernstein's inequality, and H\"older's inequality,
\begin{equation}
\begin{split}
{\rm I}^{(2.2,M)}_1 & \les \|\indh\|_{L^1_t}\|P_{N_{(1)}}v_{(1)}\|_{L^\infty_t L^2_x}\|P_{N_{(2)}}v_{(2)}\|_{L^\infty_t L^2_x}\|P_M(P_{N_{(3)}}v_{(3)}P_{N_{(4)}}v_{(4)})\|_{L^\infty_{tx}} \\
& \les M\|\indh\|_{L^1_t}\|P_{N_{(1)}}v_{(1)}\|_{L^\infty_t L^2_x}\|P_{N_{(2)}}v_{(2)}\|_{L^\infty_t L^2_x}\|P_{N_{(3)}}v_{(3)}P_{N_{(4)}}v_{(4)}\|_{L^\infty_{t}H^{-\frac12}} \\
& \les T^\theta N_{(1)}^{1-\alpha+}N_{(3)}^{-\frac12}N_{(4)}^{\frac12 +}\|P_{N_{(1)}}v_{(1)}\|_{L^\infty_t L^2_x}\|P_{N_{(2)}}v_{(2)}\|_{L^\infty_t L^2_x}\|P_{N_{(3)}}v_{(3)}\|_{L^\infty_t L^2_x} \|P_{N_{(4)}} v_{(4)} \|_{L^\infty_t L^2_x},
\end{split}
\end{equation}
with the same for ${\rm I}^{(2.2,M)}_2$; this is sufficient for \eqref{Term2-2}. Lastly, as we have done many times already, for $ {\rm I}_{3}^{(2.2,M)} $ we introduce the following modulation decomposition
\begin{equation}
\begin{split}
{\rm I}_{3}^{(2.2,M)}  & \leq {\rm I}_{3}^{(2.2,M)} (P_{N_1}Q_{\gtrsim L}\indl v_1,P_{N_2}\indl v_2, P_{N_3} v_3, P_{N_4} v_4) \\
&\quad + {\rm I}_{3}^{(2.2,M)} (P_{N_1}Q_{\ll L}\indl v_1,P_{N_2}Q_{\gtrsim L}\indl v_2, P_{N_3} v_3, P_{N_4} v_4) \\
&\quad + {\rm I}_{3}^{(2.2,M)} (P_{N_1}Q_{\ll L}\indl v_1,P_{N_2}Q_{\ll L}\indl v_2, P_{N_3}Q_{\gtrsim L} v_3, P_{N_4} v_4) \\
& \quad+ {\rm I}_{3}^{(2.2,M)} (P_{N_1}Q_{\ll L}\indl  v_1,P_{N_2}Q_{\ll L}\indl v_2, P_{N_3}Q_{\ll L} v_3, P_{N_4}Q_{\gtrsim L} v_4) \\ 
& =: {\rm I}_{3,1}^{(2.2,M)} + {\rm I}_{3,2}^{(2.2,M)}+ {\rm I}_{3,3}^{(2.2,M)}+ {\rm I}_{3,4}^{(2.2,M)}.
\end{split}
\end{equation}
We assume the case $N_{(1)} \sim N_1 \sim N_4 \gtrsim N_2,N_3$.  The
alternative ordering $N_{(1)} \sim N_2\sim N_3 \gtrsim N_1,N_4$ follows by
exchanging the two frequency pairs.
By H\"older's inequality, \eqref{indL2Bound}, $L := N_{(1)}^{\alpha-1} M \gg R$, Lemma \ref{QLlow}, and Lemma \ref{SobolevMult},
\begin{equation}
\begin{split}
{\rm I}^{(2.2,M)}_{3,1} & \les \|\indl\|_{L^2_t} \|P_{N_1}Q_{\gtrsim L}v_1 \|_{L^2_{tx}}\|P_{M}(P_{N_2}v_2 P_{N_3}v_3 ) \|_{L^\infty_{tx}} \|P_{N_4} v_4 \|_{L^\infty_t L^2_x} \\
& \les T^\frac12 N_{(1)}^{1-\alpha} \|P_{N_1}v_1\|_{X^{0,1}} \|P_{N_2}v_2 P_{N_3}v_3 \|_{L^\infty_t H^{-\frac12}} \|P_{N_4}v_4\|_{L^\infty_t L^2_x} \\
& \les T^\frac{1}{2}N_{(1)}^{1-\alpha}N_{(3)}^{-\frac12}N_{(4)}^{\frac12 +}\|P_{N_1}v_1 \|_{X^{0,1}}\|P_{N_2}v_2\|_{L^\infty_t L^2_x} \|P_{N_3}v_3\|_{L^\infty_t L^2_x} \|P_{N_4}v_4\|_{L^\infty_t L^2_x}.
\end{split}
\end{equation}
Furthermore
\begin{equation}
\begin{split}
{\rm I}^{(2.2,M)}_{3,2} &\les \|\indl\|_{L^2_t}\|P_{N_1}v_1 \|_{L^\infty_t L^2_x} \|P_{M}(P_{N_2}Q_{\gtrsim L}v_2 P_{N_3}v_3 ) \|_{L^2_{t}L^\infty_x} \|P_{N_4}v_4 \|_{L^\infty_t L^2_x} \\
& \les T^\frac12 M\|P_{N_1}v_1 \|_{L^\infty_t L^2_x} \|P_{N_2}Q_{\gtrsim L}v_2 P_{N_3}v_3  \|_{L^2_{t}H^{-\frac12}} \|P_{N_4}v_4 \|_{L^\infty_t L^2_x} \\
& \les T^\frac12 MN_{(3)}^{-\frac12} N_{(4)}^{\frac12 +}\|P_{N_1}v_1 \|_{L^\infty_t L^2_x} \|P_{N_2}Q_{\gtrsim L}v_2\|_{L^2_{tx}} \| P_{N_3}v_3  \|_{L^\infty_t L^2_x} \|P_{N_4}v_4 \|_{L^\infty_t L^2_x} \\
& \les T^\frac{1}{2}N_{(1)}^{1-\alpha}N_{(3)}^{-\frac12}N_{(4)}^{\frac12 +}\|P_{N_1}v_1 \|_{L^\infty_t L^2_x}\|P_{N_2}v_2\|_{X^{0,1}} \|P_{N_3}v_3\|_{L^\infty_t L^2_x} \|P_{N_4}v_4\|_{L^\infty_t L^2_x}.
\end{split}
\end{equation}
Both bounds are sufficient for \eqref{Term2-2}.  For
${\rm I}_{3,3}^{(2.2,M)}$ and ${\rm I}_{3,4}^{(2.2,M)}$, the same calculation
places the $X^{0,1}$ norm on $v_3$ and $v_4$, respectively.  This completes
the proof of \eqref{Term2-2}.

\medskip 

\subsection{Second term estimate}~\\

We now consider the second term, \eqref{DiffTerm2}. After frequency decomposition, we are required to estimate
\begin{equation} \label{term2}
{\rm I\!I} :=\sum_{N \gg 1} \sum_{N_1,N_2,N_3,N_4} 2N^{2(s-1)}\Im \bigg[ \int_0^T \int_\T P_{N_1}u P_{N_2}\bar{w} P_{N_3}v D_x^{2\beta}P_N^2 P_{N_4}\bar{w} \, dxdt\bigg].
\end{equation}
Note that no symmetry can be abused to give rise to any commutators, though since $u$ and $v$ play symmetric roles we need not distinguish between them. We still subdivide into the following cases
\begin{itemize}
    \item {\it Case 1}: $|\Omega| \gtrsim \Nmax^\alpha$
    \item {\it Case 2}: $1\le |\Omega| \ll \Nmax^{\alpha} $
\end{itemize}
For {\it Case 1}, the estimate \eqref{Term1-1} applies after permuting the
four inputs so that the two $w$ factors occupy the third and fourth slots.
Multiplying by the Sobolev weights in \eqref{term2} gives the same dyadic
sum as in \eqref{DiffTerm2Dyadic}, with the stronger factor
$N_{(1)}^{-\alpha}$ in place of $N_{(1)}^{1-\alpha}$.  Hence the summation
below also controls {\it Case 1}.
As for {\it Case 2}, we may assume that $N_{(4)} \ll N_{(1)}$ in view of \eqref{DiffCase1}, furthermore we may subdivide into two cases
\[ |k_1 - k_2| \ll N_{(1)} \text{ and } |k_2 - k_3| \gtrsim N_{(1)}\]
\[ |k_1 - k_2| \gtrsim N_{(1)} \text{ and } |k_2 - k_3| \ll N_{(1)}.\]
The two subcases are symmetric after interchanging $u$ and $v$, so it is
enough to consider the second one.  Applying \eqref{Term2-2}, with the
indices permuted to match the sign pattern in \eqref{term2}, gives
\begin{equation}\label{DiffTerm2Dyadic}
\begin{split}
|{\rm I\!I}|
&\lesssim T^\theta
\sum_{\substack{N_1,N_2,N_3,N_4\ge1\\N_{(1)}\sim N_{(2)}}}
N_4^{s-1+2\beta}N_2^{1-s}N_1^{-s}N_3^{-s}
N_{(1)}^{1-\alpha+2\beta}
N_{(3)}^{-1/2}N_{(4)}^{1/2+} \\
&\qquad\times
\|P_{N_1}u\|_{Z_T^s}\|P_{N_3}v\|_{Z_T^s}
\Big(
\|P_{N_2}w\|_{Z_T^{s-1}}\|P_{N_4}w\|_{L^\infty_TH^{s-1}}
+ \|P_{N_2}w\|_{L^\infty_TH^{s-1}}
  \|P_{N_4}w\|_{Z_T^{s-1}}
\Big).
\end{split}
\end{equation}
We briefly verify the summation. If $N_4\sim N_{(1)}$, the two largest
frequencies and the condition $|k_2 - k_3| =|k_4 - k_1| \ll N_{(1)}$ force $N_1 \sim N_4 \sim N_{(1)}$. If furthermore $N_2 \sim N_{(1)}$ the same condition would imply $N_3 \sim N_{(1)}$ and we return to the special case \eqref{DiffCase1}. 
Otherwise the frequency contribution
is bounded by
\begin{equation}
N_{(1)}^{4\beta - \alpha} N_2^{1-s} N_{(3)}^{-\frac12} N_{(4)}^{\frac12 - s+} \lesssim N_{(1)}^{4\beta - \alpha} N_{(3)}^{\frac12 -s}N_{(4)}^{\frac12- s+},
\end{equation}
which is summable for $s > 1/2$ and $\alpha \geq 4\beta$.
If $N_4\ll N_{(1)}$, then $|k_4 - k_1| \ll N_{(1)}$ forces $N_1 \ll N_{(1)}$. Hence $N_2 \sim N_3 \sim N_{(1)}$; collecting factors gives
\[ N_{(1)}^{2\beta + 2 - \alpha -2s} N_4^{s-1 + 2\beta}N_1^{-s} N_{(3)}^{-\frac12} N_{(4)}^{\frac12 +}.\]
If $N_{(4)} \sim N_1$ and $N_{(3)} \sim N_4$, summing first in $N_{(4)}$ and then in $N_{(3)}$
\[
N_{(1)}^{2-2s-\alpha+2\beta}
N_{(3)}^{s-\frac32+2\beta}.
\]
When $s<\frac32-2\beta$, this is summable under
$s\ge\beta+1-\alpha/2$; when $s\ge\frac32-2\beta$, it is summable under
$s>1/2$ and $4\beta\le\alpha$. Otherwise, if $N_{(4)} \sim N_4$ and $N_{(3)} \sim N_1$ we get, provided $s \geq \beta + 1-\frac\alpha2$,
\[ N_{(1)}^{2\beta + 2-\alpha -2s}N_{(3)}^{\frac12 - s}N_{(4)}^{\frac12 -s+} N_{(3)}^{-1} N_{(4)}^{2(s+\beta)-1} \les  N_{(4)}^{4\beta -\alpha}N_{(3)}^{\frac12 - s}N_{(4)}^{\frac12 -s+}, \]
which again is summable for $s > 1/2$ and $\alpha \geq 4\beta$.
Thus \eqref{DiffTerm2Dyadic} is bounded by
\[
T^\theta
\|w\|_{L^\infty_TH^{s-1}}\|w\|_{Z_T^{s-1}}
\big(\|u\|_{Z_T^s}+\|v\|_{Z_T^s}\big)^2.
\]
Combining this with the estimates for \eqref{DiffTerm1}, and enlarging the
last factor harmlessly, proves the proposition.
\end{proof}

\section{Proof of Theorem \ref{MainTheorem}}
\label{sec-7}

 In this section, we use the estimates derived in the previous sections to prove the unconditional local well-posedness result stated in Theorem \ref{MainTheorem}.
 The argument follows the frequency-envelope compactness framework in
 \cite{MolinetTanakaUncond} and the references therein.  We include the
 details needed for the present equation.

When $\beta=0$, the result follows directly from the Duhamel formula:
$H^s(\T)$ is an algebra for $s>1/2$, the linear group is unitary on $H^s$,
and the cubic nonlinearity is locally Lipschitz on $H^s$. A standard
contraction in $C([0,T];H^s)$ therefore gives existence, unconditional
uniqueness, and continuous dependence. In the remainder of this section we
assume $\beta>0$.

\begin{proof}[Proof of local existence and strong continuity]
Fix $u_0\in H^s$ and set $A=\|u_0\|_{H^s}$. Let $\Pi_K$ denote the
Fourier projection onto $|k|\le K$, and consider the Galerkin system
\begin{equation}\label{Galerkin}
\begin{cases}
i\partial_tu_K+D_x^\alpha u_K
=\nu\Pi_KD_x^{2\beta}(|u_K|^2u_K),\\
u_K(0)=\Pi_Ku_0.
\end{cases}
\end{equation}
This is a finite-dimensional ODE and therefore has a unique smooth solution
on a maximal interval. Since all multipliers in the preceding estimates
commute with $\Pi_K$, Propositions \ref{Apriori} and \ref{DiffEstimate}, as
well as Lemma \ref{BourgainEstimates}, hold uniformly for \eqref{Galerkin}.

With $\omega_N\equiv1$, these estimates give, for some $\theta>0$,
\begin{equation}\label{GalerkinBootstrap}
\|u_K\|_{L^\infty_TH^s}^2
\le A^2+
CT^\theta\|u_K\|_{L^\infty_TH^s}^4
\big(1+\|u_K\|_{L^\infty_TH^s}^2\big)^6.
\end{equation}
Choose $T=T(A)\le1$ so that
\[
CT^\theta(2A+1)^2\big(1+(2A+1)^2\big)^6\le\frac14.
\]
A standard continuity argument in \eqref{GalerkinBootstrap} then gives
\begin{equation}\label{uniformHsBound}
\sup_K\|u_K\|_{L^\infty_TH^s}\le 2A+1.
\end{equation}
In particular, no Galerkin solution can cease to exist before $T$.

Choose an admissible frequency envelope $\omega_N\to\infty$ such that
$u_0\in H^s_\omega$; see \cite[Lemma 4.1]{KT-03}. Repeating the preceding
bootstrap with Proposition \ref{Apriori} in $H^s_\omega$, while using
\eqref{uniformHsBound} for all unweighted factors, yields
\begin{equation}\label{uniformWeightedBound}
\sup_K\|u_K\|_{L^\infty_TH^s_\omega}
\lesssim_{A,T}\|u_0\|_{H^s_\omega}.
\end{equation}

For completeness, we record the approximation point needed below. Applying
the proof of Proposition \ref{DiffEstimate} to two Galerkin solutions gives
\begin{equation}\label{GalerkinDifference}
\begin{split}
\|u_K-u_L\|_{L^\infty_TH^{s-1}}^2
&\lesssim \|\Pi_Ku_0-\Pi_Lu_0\|_{H^{s-1}}^2 +T^\theta C_A
\|u_K-u_L\|_{L^\infty_TH^{s-1}}^2+\varepsilon_{K,L},
\end{split}
\end{equation}
where $\varepsilon_{K,L}\to0$ as $K,L\to\infty$. Indeed, after subtracting
the two Galerkin equations, the new terms contain
$(\Pi_K-\Pi_L)D_x^{2\beta}(|u_J|^2u_J)$ for $J=K$ or $L$. In every dyadic
estimate for such a term the output frequency is $\gtrsim K\wedge L$, and
the relation $k_1-k_2+k_3=k_4$ then forces at least one input frequency to
be $\gtrsim K\wedge L$. By \eqref{uniformWeightedBound}, the corresponding
$H^s$ factor is $O(\omega_{K\wedge L}^{-1})$. Thus the same summations as in
Proposition \ref{DiffEstimate} give
$\varepsilon_{K,L}\lesssim_A\omega_{K\wedge L}^{-2}$.
After decreasing $T(A)$ so that the second term on the right of
\eqref{GalerkinDifference} can be absorbed, $(u_K)$ is Cauchy in
$C([0,T];H^{s-1})$.

The weighted estimate upgrades this convergence to the endpoint space.
For every dyadic $M$,
\[
\|u_K-u_L\|_{L^\infty_TH^s}
\lesssim M\|u_K-u_L\|_{L^\infty_TH^{s-1}}
+\omega_M^{-1}
\sup_J\|u_J\|_{L^\infty_TH^s_\omega}.
\]
First let $K,L\to\infty$ and then $M\to\infty$. Hence $u_K$ converges to
some $u$ in $C([0,T];H^s)$. Since $s>1/2$, the cubic map is continuous on
$H^s$, and passing to the limit in \eqref{Galerkin} shows that $u$ satisfies
\eqref{MMTEq} in the distributional sense and that $u(0)=u_0$. This proves
both existence and strong $H^s$ continuity.
\end{proof}

\begin{proof}[Proof of unconditional uniqueness]
Let $u,v\in L^\infty([0,T];H^s)$ be two solutions with the same initial
data, and put $w=u-v$. Lemma \ref{BourgainEstimates} gives
$u,v\in Z_T^s$ and
\[
\|w\|_{Z_\tau^{s-1}}\lesssim_{u,v}\|w\|_{L^\infty_\tau H^{s-1}}
\qquad (0<\tau\le T).
\]
Proposition \ref{DiffEstimate} therefore implies
\[
\|w\|_{L^\infty_\tau H^{s-1}}^2
\le C_{u,v}\tau^\theta\|w\|_{L^\infty_\tau H^{s-1}}^2.
\]
Choosing $\tau>0$ so that $C_{u,v}\tau^\theta<1$ gives $w=0$ on
$[0,\tau]$. The estimates are invariant under time translation, and the
relevant norms remain bounded on $[0,T]$; finitely many repetitions cover
the whole interval.
\end{proof}

\begin{proof}[Proof of continuity with respect to initial data]
Let $u_0^n\to u_0$ in $H^s$, and let $u_n,u$ be the corresponding
solutions on a common interval $[0,T]$, where $T$ depends only on a uniform
$H^s$ bound for the data. The set
\[
\mathcal K=\{u_0\}\cup\{u_0^n:n\ge1\}
\]
is compact in $H^s$. Consequently there is a single admissible envelope
$\omega_N\to\infty$ such that
\begin{equation}\label{CommonEnvelope}
\sup_{f\in\mathcal K}\|f\|_{H^s_\omega}<\infty;
\end{equation}
see \cite[Lemma 4.1]{KT-03} and
\cite[Lemma 4.6]{MolinetTanakaUncond}. The weighted {\it a priori} estimate and
the bootstrap above imply
\begin{equation}\label{HfreqEnvelopeBound}
\sup_{n\ge1}\|u_n\|_{L^\infty_TH^s_\omega}
+\|u\|_{L^\infty_TH^s_\omega}\le C.
\end{equation}
On the other hand, Proposition \ref{DiffEstimate} and
Lemma \ref{BourgainEstimates} give
$u_n\to u$ in $C([0,T];H^{s-1})$. Therefore, for every dyadic $M$,
\[
\|u_n-u\|_{L^\infty_TH^s}
\lesssim
M\|u_n-u\|_{L^\infty_TH^{s-1}}
+\frac{C}{\omega_M}.
\]
Letting first $n\to\infty$ and then $M\to\infty$ proves convergence in
$C([0,T];H^s)$ and hence continuity of the solution map.
\end{proof}

\bigskip
\subsection*{Acknowledgment} The first author would like to thank FAPESP Brazil for financial support under grant (\#2024/10613-4) and the School of Mathematics, University of Birmingham, UK, for hospitality where this work was developed. Y.W. was supported by the EPSRC Mathematical Sciences Small Grant (grant no. UKRI1116).\\


\noindent
{\bf Conflict of interest statement.} 
On behalf of all authors, the corresponding author states that there is no conflict of interest.\\

\noindent 
{\bf Data availability statement.} 
The datasets generated during and/or analyzed during the current study are available from the corresponding author on reasonable request.



\begin{thebibliography}{99}


\bibitem{FracLeib} A. B\'enyi, T. Oh,  T. Zhao; {\em Fractional Leibniz rule on the torus,}  Proc. Amer. Math.
Soc., {\bf 153} (1) (2025) 207--221.

\bibitem{Bourgain-93}  J.  Bourgain; {\em  Fourier transform restriction phenomena for certain lattice subsets and applications to
nonlinear evolution equations i. {S}chr\"odinger equations,}  Geom. Funct. Anal., {\bf 3} (2) (1993) 107--156.

\bibitem{BLLZ-23} E. Brun, G. Li, R. Liu, Y. Zine {\em Global well-posedness of one-dimensional cubic fractional nonlinear Schr\"odinger equations in negative Sobolev spaces,}
arXiv:2311.13370

\bibitem{YGSS-WellandIll} Y. Cho, G. Hwang, S. Kwon, S. Lee.;{\em Well-posedness and ill-posedness for the cubic fractional {S}chr\"odinger equations,} Discrete Contin. Dyn. Syst. {\bf 35} (7) (2015) 2863--2880

\bibitem {DET-ExistTheoryfNLS}  S. Demirbas, M. Erdo\u gan, N. Tzirakis {\em Existence and uniqueness theory for the fractional {S}chr\"odinger equation on the torus,} Some topics in harmonic analysis and applications. {Adv. Lect. Math. (ALM)} {\bf 34} (2016) 145–-162

\bibitem{FS-22} J. Forlano, K. Seong; {\em Transport of {G}aussian measures under the flow of one-dimensional fractional nonlinear {S}chr\"odinger equations,} Comm. Partial Differential Equations {\bf 47} (6) (2022) 1296--1337

\bibitem{GTV} J. Ginibre, Y. Tsutsumi, G. Velo; {\em On the Cauchy problem for the Zakharov system,} J. Funct. Anal. {\bf 151} (2) (1997) 384--436.

\bibitem{HIKK-10} S. Herr, A.D. Ionescu, C.E. Kenig, H. Koch; {\em  A para-differential renormalization technique for non-linear dispersive equations,} Commun. Partial Differ. Equ. {\bf 35} (10) (2010) 1827--1875.

\bibitem{IKT-08} A.D. Ionescu, C.E. Kenig, D. Tataru; {\em Global well-posedness of the initial value problem for the KP-I equation in the energy space,} Invent. Math. {\bf 173} (2) (2008) 265--304.

\bibitem{KKO-25} T. Kato, T. Kondo, M. Okamoto; {\em Well- and Ill-posedness of the Cauchy problem for derivative fractional nonlinear Schrödinger equations on the torus,} 	arXiv:2508.11866

\bibitem{KK-03} C. Kenig, K. Koenig; {\em  On the local well-posedness of the Benjamin-Ono and modified Benjamin-Ono equations,}  Math. Res. Lett. {\bf 10} (5-6) (2003) 879--895.

\bibitem{Kishimoto} N. Kishimoto {\em Unconditional uniqueness of solutions for nonlinear dispersive equations,} arXiv:1911.04349

\bibitem{KT-07} H. Koch, D. Tataru; {\em  A priori bounds for the 1D cubic NLS in negative Sobolev spaces,}  Int. Math. Res. Not. (16) (2007) rnm053

\bibitem{KT-03} H. Koch, N. Tzvetkov; {\em On the local well-posedness of the Benjamin-Ono equation in $ H^s(\R)$,}  Int. Math. Res. Not. (26) (2003) 1449--1464.

\bibitem{TKMO-26} T. Kondo, M. Okamoto; {\em Norm inflation for quadratic derivative fractional nonlinear Schr\"odinger equations,} J. Evol. Equ. {\bf 26} (2) (2026) Paper No. 74.

\bibitem{KO-25} T. Kondo, M. Okamoto; {\em Well- and ill-posedness of the Cauchy problem for semi-linear Schr\"odinger equations on the torus,} 	arXiv:2501.04205


\bibitem{MMT-97} A. Majda, D. McLaughlin, E.  Tabak; {\em  A one-dimensional model for dispersive wave turbulence,} J. Nonlinear Sci. {\bf 7} (1) (1997) 9--44.


\bibitem{MolinetUncondKdV}  L. Molinet, D. Pilod, S. Vento; {\em On unconditional well-posedness for the periodic modified
Korteweg–de Vries equation,}  J. Math. Soc. Japan, {\bf 71} (1) (2019) 147--201.

\bibitem{MPV-18} L. Molinet, D. Pilod, S. Vento; {\em  On well-posedness for some dispersive perturbations of Burgers’ equation,}  Ann. Inst. Henri Poincaré, Anal. Non Linéaire {\bf 35} (7) (2018) 1719--1756.

\bibitem{MT-25} L. Molinet, T. Tanaka; {\em Local well-posedness for the derivative nonlinear Schr\"odinger equation with nonvanishing boundary conditions,} arXiv:2025.20883v1

\bibitem{MolinetTanakaUncond} L. Molinet, T. Tanaka; {\em Unconditional well-posedness for some nonlinear periodic one-dimensional
dispersive equations,} J. Funct. Anal. {\bf 283} (1) (2022) Paper No. 109490, 45 pp.

\bibitem{MST-01} L. Molinet, J.-C. Saut, N. Tzvetkov; {\em  Ill-posedness issues for the Benjamin-Ono and related equations,} SIAM J. Math. Anal. {\bf 33} (4) (2001) 982--988.

\bibitem{MV-15} L. Molinet, S. Vento; {\em  Improvement of the energy method for strongly nonresonant dispersive equations and applications,} Anal. PDE {\bf 8} (6) (2015) 1455--1495.


\bibitem{PPW-26} M. Panthee, J. Patterson, Y. Wang; {\em On the well-posedness of the initial value problem for the MMT model,}  arXiv:2601.07771

\bibitem{TzvetkozPeriodicKP1} J.-C. Saut,  N. Tzvetkov; {\em  On periodic KP-I type equations,}  Comm. Math. Phys., {\bf 221} (3) (2001) 451--476.

\bibitem{ST-2020} C. Sun, N. Tzvetkov, {\em Gibbs measure dynamics for the fractional {NLS},}
SIAM J. Math. Anal. {\bf 52} (5) (2020) 4638--4704.

\bibitem{HT-99} H. Takaoka; {\em Well-posedness for the one-dimensional nonlinear Schr\"odinger equation with the derivative nonlinearity,} Adv. Differ. Equ., {\bf 4} (4) (1999) 561--580.

\bibitem{TT-04} T. Tao; {\em  Global well-posedness of the Benjamin-Ono equation in $H^1(\R)$,}  J. Hyperbolic Differ. Equ. {\bf 1} (1) (2004) 27--49.

\bibitem{TTaoL2} T.  Tao; {\em Multilinear weighted convolution of $ L^2$-functions, and applications to nonlinear dispersive
equations,}  Amer. J. Math., {\bf 123} (5) (2001) 839--908.

\bibitem{ZDP-04}  V. Zakharov, F. Dias, A. Pushkarev;  {\em One-dimensional wave turbulence,} Phys. Rep. {\bf 398} (1) (2004) 1--65.

\bibitem{ZGPD-01}  V. Zakharov, P. Guyenne, A. Pushkarev, F. Dias;  {\em Wave turbulence in one-dimensional models,}  Phys. D. {\bf 152/153} (2001) 573--619.



\end{thebibliography}
\end{document}